\documentclass{amsart}
\usepackage{geometry}                
\usepackage{graphicx}
\usepackage{amssymb}
\usepackage{epstopdf,amsmath}
\newtheorem{prop}{Proposition}[section]
\newtheorem{lemma}[prop]{Lemma}
\newtheorem{cor}[prop]{Corollary}

\newtheorem{theorem}[prop]{Theorem}
\newtheorem{hypoth}[prop]{Hypotheses}

\DeclareMathOperator{\Vol}{Vol}
\DeclareMathOperator{\Dim}{Dim}
\DeclareMathOperator{\Poly}{Poly}
\DeclareMathOperator{\Deg}{Deg}
\DeclareMathOperator{\Avg}{Avg}
\DeclareMathOperator{\Angle}{Angle}
\DeclareMathOperator{\Length}{Length}
\DeclareMathOperator{\Area}{Area}
\DeclareMathOperator{\Dist}{Dist}
\DeclareMathOperator{\Prob}{Prob}

\newtheorem{definition}[prop]{Definition}
\newtheorem{reascube}{Reasonable Cube Condition}
\newtheorem{reasseg}{Reasonable Tube Segment Condition}
\newtheorem{reasslice}{Reasonable Slice Condition}

\newcommand{\FF}{\mathbb{F}}

\newcommand{\RR}{\mathbb{R}}

\title{Degree reduction and graininess for Kakeya-type sets in $\RR^3$}

\author{Larry Guth}

\begin{document}

\maketitle

\begin{abstract}  Let $\frak T$ be a set of cylindrical tubes in $\RR^3$ of length $N$ and radius 1.  If the union of the tubes has volume $N^{3 - \sigma}$, and each point in the union lies in tubes pointing in three quantitatively different directions, and if a technical assumption holds, then at scale $N^\sigma$, the tubes are clustered into rectangular slabs of dimension $1 \times N^\sigma \times N^\sigma$.  This estimate generalizes the graininess estimate in \cite{KLT}.  The proof is based on modeling the union of tubes with a high-degree polynomial.

\end{abstract}

In \cite{D}, Dvir proved the finite field Kakeya conjecture using the polynomial method.  It is an interesting open problem how much this approach can tell us about the Kakeya problem in $\RR^n$.  
The paper \cite{GK} uses the polynomial method to prove results about the combinatorics of finite sets of lines in $\RR^3$.  The Kakeya problem involves thin tubes instead of lines, and it seems to be quite difficult to adapt the polynomial method from lines to tubes.  In this paper, we adapt some of the ideas from \cite{GK} to prove results about tubes in $\RR^3$.  Our results describe some structural features of a (hypothetical) Kakeya set.

The paper \cite{GK} proves that a set of lines with too many high-multiplicity intersections must cluster into planes.  Here is a precise statement (this is Theorem 1.2 in \cite{GK}.)

\begin{theorem} \label{gk} There exists a constant $c > 0$ so that the following holds.  Suppose that $\frak
 L$ is a set of $N^2$ lines in $\RR^3$.  Suppose that $X$ is a set of points in $\RR^3$, and each line of $\frak L$ contains at least $N$ points of $X$.  If $|X| < c N^3$, then there is a plane that contains at least $N+1$ lines of $\frak L$.
\end{theorem}

In this paper, we will prove a theorem about tubes in the spirit of Theorem \ref{gk}.  However, our theorem about tubes is weaker in an important sense.  We will prove that a set of tubes with too many high-multiplicity intersections must cluster into planes {\it when restricted to balls of an appropriate radius}.

Here is a rough statement of our main theorem.  Suppose that $\frak T$ is a set of cylinders in $\RR^3$ with length $N$ and radius 1.  Suppose that the union of the cylinders in $\frak T$ has volume $N^{3 - \sigma}$, and suppose that most points in the union are contained in three tubes of $\frak T$ pointing in quantitatively different directions.  Then in a typical ball of radius $N^\sigma$, the union of the tubes resembles a collection of rectangular slabs of dimensions $1 \times N^\sigma \times N^\sigma$.  

This type of estimate is called a graininess estimate.  The first graininess estimate was proven by Katz, Laba, and Tao in \cite{KLT}.   We will recall some of their work in the next subsection.

Our proof is based on finding a polynomial surface of controlled degree that models the union of the tubes of $\frak T$.  We will find such a polynomial surface with degree $\lesssim N^{1- \sigma}$, and this degree estimate is optimal. 

\subsection{Planiness and graininess}

The paper \cite{KLT} proves that, for small $\epsilon$, a Kakeya set of Minkowski dimension $(5/2) + \epsilon$ in $\RR^3$ must have three remarkable structural properties: stickiness, planiness, and graininess.  Combining these properties with number theoretic arguments from \cite{B}, \cite{KLT} derives a contradiction for sufficiently small $\epsilon$.  In this way, they prove that a Kakeya set in $\RR^3$ must have upper Minkowski dimension at least $(5/2) + \epsilon$ for a small positive $\epsilon$.   For context, we recall the rough statements of their results on planiness and graininess.

Suppose that $\frak T$ is a set of cylindrical tubes in $\RR^3$ of length $N$ and radius 1.  The direction of a tube $T$ is the unit vector parallel to the central line of $T$.  We write $v(T)$ for the direction of $T$.  We say that $\frak T$ is a Kakeya set of tubes if it obeys the following hypotheses.

\begin{itemize}

\item There are $N^2$ tubes in $\frak T$.

\item For any two different tubes $T_i, T_j \in \frak T$, the angle between $v(T_i)$ and $v(T_j)$ is $\gtrsim 1/N$.

\end{itemize}

\cite{KLT} studies a Kakeya set of tubes where the volume of the union of the tubes is $\lesssim N^{(5/2) + \epsilon}$ for a small $\epsilon > 0$.  Their results also require assumptions at other scales: they also assume a volume bound on the union of the concentric tubes of radius $N^{1/2}$.  Since our paper doesn't involve any multi-scale considerations, we omit the details.  Under these assumptions, the authors prove that the set of tubes must be plany and grainy.  

Planiness roughly means that all the tubes of $\frak T$ that intersect a typical unit cube $Q$ lie close to a plane.  For each unit cube $Q$ that intersects the union of the tubes, they can assign a plane $\pi(Q)$, and for almost all $Q$, for almost all the tubes $T \in \frak T$ that intersect $Q$, the angle between $v(T)$ and $\pi(Q)$ is at most (roughly) $N^{-1/2}$. 

Graininess roughly means that the restriction of $\frak T$ to a typical ball of radius $N^{1/2}$ consists of parallel rectangular slabs of dimension $1 \times N^{1/2} \times N^{1/2}$.  Within this typical ball, the planes $\pi(Q)$ are all parallel to these slabs, and so they all agree up to an angle $\sim N^{-1/2}$.

In particular, if $Q, Q'$ lie in the same tube $T \in \frak T$, and the distance from $Q$ to $Q'$ is less than $N^{1/2}$, then the angle between $\pi(Q)$ and $\pi(Q')$ is (almost always) $\lesssim N^{-1/2}$.  
This estimate about how $\pi(Q)$ rotates as we slide $Q$ along a tube $T$ is the estimate that we will generalize.  This bound is only part of the graininess estimate in \cite{KLT}.  It forces the tubes in a typical ball of radius $N^{1/2}$ to organize into (disjoint) $1 \times N^{1/2} \times N^{1/2}$ slabs, but it doesn't force the slabs to be parallel.

One limitation of the proof in \cite{KLT} is that it only works for Kakeya sets of dimension close to $5/2$.  The dependence on $\epsilon$ goes as follows (see Proposition 8.1 in \cite{KLT}) :  If the dimension of the Kakeya set is $(5/2) + \epsilon$, then the angle in the planiness estimate is bounded by $N^{-1/2} N^{C \sqrt{\epsilon}}$ for a (fairly large) constant $C$.  When $C \sqrt{\epsilon} > 1/2$, then the planiness estimate becomes vacuous.  The story for graininess is similar.

The planiness estimate was proven in a different way in \cite{BCT}.  The multilinear Kakeya inequality in that paper is a very useful generalization of planiness.  For example, it shows that for any $\sigma > 0$, for a Kakeya set of tubes in $\RR^3$ with volume $N^{3 - \sigma}$, most tubes through a typical unit cube $Q$ lie within an angle $N^{-\sigma}$ of a plane $\pi(Q)$.  The multilinear Kakeya inequality was reproven (and slightly strengthened) in \cite{Gu1}, using the polynomial method.

In this paper we give a different approach to graininess using the polynomial method.  In some ways, our graininess result is more general than the one in \cite{KLT}, but it is also weaker in some other ways.   We state our main theorem precisely in the next subsection.

\subsection{Statement of results}

We will work with sets of tubes obeying the following hypotheses:

\begin{hypoth}  \label{uniform3transintro} Let $E > 1$.  
Suppose that $\frak T$ is a set of tubes in $\RR^3$ with radius 1 and length $EN$, contained in a ball of radius $EN$.  Suppose that $X$ is a set of $N^{3 - \sigma}$ disjoint unit cubes in this ball.  Suppose that $X$ and $\frak T$ obey the following conditions:

\begin{enumerate}

\item Each tube $T \in \frak T$ intersects between $N$ and $E N$ cubes of $X$.

\item Each cube of $X$ intersects between $\rho$ and $E \rho$ tubes of $\frak T$, for some $\rho \ge 3$.

\item Each point of $\RR^3$ lies in $\le E \rho$ tubes of $\frak T$.

\item (At least three directions of tubes at each point) For each cube $Q \in X$, and for any two unit vectors $v_1, v_2 \in \RR^3$, at least a fraction $E^{-1}$ of the tubes of $\frak T$ that intersect $Q$ have angle $\ge E^{-1}$ with both $v_1$ and $v_2$.

\end{enumerate}

\end{hypoth}

Our results will be interesting when $E$ is much smaller than $N$: the reader may take $E = 100$ as a good special case.

Hypothesis (1) says that $X$ covers a significant fraction of each tube $T \in \frak T$.  Hypotheses (2) and (3) say that the density of tubes is uniform over the set $X$, and also that $X$ is the region of highest density.  These are technical hypotheses, and it may be possible to weaken or remove them.  Hypothesis (4) says that the tubes through a given $Q \in X$ point in at least three different directions in a quantitative sense.  This is a crucial hypothesis as we will see below.

Let's compare these hypotheses to the hypotheses for a Kakeya set. 
In Hypotheses \ref{uniform3transintro}, we don't need to assume that the number of tubes is $N^2$, and we don't need to assume that the tubes point in different directions.  We assume instead some uniformity, and we assume that the tubes through a given cube point in at least three directions.  A Kakeya set does not necessarily obey Hypotheses \ref{uniform3transintro}, but I hope that these additional hypotheses are fairly minor.  On the other hand, there are sets of tubes that are not Kakeya sets but which obey Hypotheses \ref{uniform3transintro}.  We will give a simple example later in the introduction.

Under these hypotheses, we will prove planiness and graininess estimates in the spirit of \cite{KLT}.  Planiness says that for a typical cube $Q \in X$, most of the tubes of $\frak T$ through $Q$ lie near to a plane $\pi(Q)$.  
We will prove the following planiness estimate using the polynomial method:

\begin{prop} \label{planypropintro} Assume Hypotheses \ref{uniform3transintro}.  Let $\epsilon > 0$.  For each cube $Q \in X$, we can choose a plane $\pi(Q)$ through $Q$, so that for a fraction $(1 - \epsilon)$ of cubes $Q \in X$, for a fraction $(1 - \epsilon)$ of the tubes $T \in \frak T$ that meet $Q$, $\Angle(v(T), \pi(Q)) \le \Poly(E, \epsilon^{-1}) N^{-\sigma}$.
\end{prop}

This Proposition could also be proven using the multilinear Kakeya estimates in \cite{BCT} (or \cite{Gu1}), but we will give a slightly different proof below.

Our main result controls how the plane $\pi(Q)$ rotates as we vary $Q$ within a segment of a tube $T$.  

\begin{theorem} \label{graininessintro} Assume Hypotheses \ref{uniform3transintro}.  Let $\epsilon > 0$.  
Also assume that $N^\sigma$ is larger than some large polynomial in $E, \epsilon^{-1}$.
Then there is a large constant $K = \Poly(E, \epsilon^{-1})$ so that the following holds.
For a fraction $(1 - \epsilon)$ of intersecting pairs $(Q,T) \in X \times \frak T$, for a fraction $(1 - \epsilon)$ of the cubes $Q'$ of $X$ which intersect $T$ with $\Dist(Q,Q') \le K^{-1} N^\sigma$,

$$ \Angle (\pi(Q), \pi(Q')) \le K N^{- \sigma}. $$

\end{theorem}

This control of $\pi(Q)$ forces tubes to cluster into slabs of dimensions $1 \times N^\sigma \times N^\sigma$.  Consider a typical $T \in \frak T$ and a segment $Seg \subset T$ of length $\sim N^\sigma$ containing $\sim N^\sigma$ cubes of $X$.  Let $Q$ be one of these cubes, and consider a slab $Slab$ with dimensions $1 \times N^\sigma \times N^\sigma$, parallel to $\pi(Q)$, and containing $Seg$.  Almost all tubes of $\frak T$ through $Q$ must lie in this slab for length $\sim N^\sigma$.  Theorem \ref{graininessintro} says that $\pi(Q')$ is $N^{-\sigma}$-close to $\pi(Q)$ for the other $Q'$ in $Seg$.  Therefore, almost all the tubes of $\frak T$ that pass through $Seg$ lie in $Slab$ for a length $\sim N^\sigma$.  Moreover, if $T_1$ is another (typical) tube that passes through $Seg$, and $Seg_1$ is the intersection of $T_1$ with our slab, then Theorem \ref{graininessintro} says that $\pi(Q_1)$ is $N^{-\sigma}$ close to $\pi(Q)$ for almost all $Q_1 \in Seg_1$, and so almost all the tubes through $Seg_1$ also lie in $Slab$ for a length $\sim N^\sigma$.  This slab is sometimes called a grain for the set of tubes $\frak T$.  

Let us compare Theorem \ref{graininessintro} with the graininess estimate in \cite{KLT}.  In some ways, Theorem \ref{graininessintro} is more general.  
It applies to sets of tubes with total volume $N^{3 - \sigma}$ for any $\sigma > 0$.  It involves hypotheses only at one scale instead of hypotheses at several scales.  It also applies to some sets of tubes that don't point in different directions.  On the other hand, it does have some technical assumptions about the uniformity of the density of tubes, which are not needed in \cite{KLT}.  Moreover, the graininess estimate of \cite{KLT} proves something stronger.  It proves that in a typical ball of radius $\sim N^{1/2}$, the Kakeya set resembles a set of  {\it parallel} slabs of dimension $1 \times N^{1/2} \times N^{1/2}$.  We can't prove that nearby slabs are parallel, because we are only able to control how $\pi(Q)$ varies as we move $Q$ along a tube of $\frak T$.

\subsection{Degree reduction}

The proof of Theorem \ref{graininessintro} uses the polynomial method.  We find a polynomial $P$ of controlled degree whose zero set $Z(P)$ is a good model for the set of cubes $X$, and then we use $Z(P)$ to study the tubes and cubes.  Here is a precise statement about the existence of a polynomial of controlled degree that models $X$.

\begin{theorem} \label{degredtubeintro} Assume Hypotheses \ref{uniform3transintro}.   Let $\epsilon > 0$.  Then there is a non-zero polynomial $P$ of degree $\le \Poly(E, \epsilon^{-1}) N^{1 - \sigma}$, so that for $(1 - \epsilon) |X|$ cubes $Q \in X$, the area of $Z(P) \cap Q$ is at least 1.
\end{theorem}

This degree estimate is sharp up to a constant factor.  The Crofton formula says that the area of $Z(P)$ in a ball of radius $R$ is $\lesssim (\Deg P) R^2$.  Therefore, if $X$ is any set of $N^{3 - \sigma}$ disjoint unit cubes in a ball of radius $\sim N$, and if $Z(P)$ has area at least $ 1$ in most cubes of $X$, then $\Deg P \gtrsim N^{1 - \sigma}$.  

Theorem \ref{degredtubeintro} says that Kakeya-type sets can be modelled by a polynomial of the lowest plausible degree.  In other words, they have a lot of algebraic structure.  We try to exploit this algebraic structure to control the geometry of the tubes.  We are able to get a lot of information about the geometry at scales $\lesssim N^{\sigma}$, proving Theorem \ref{graininessintro}.  

\subsection{Simple examples}

Let's consider a couple examples to illustrate our results.  First suppose that $X$ is a set of unit cubes tiling a rectangular slab of dimensions $N^{1 - \sigma} \times 2N \times 2N$.  There are many tubes that intersect at least $N$ cubes of $X$.  It's not hard to choose a family $\frak T$ of such tubes so that each cube of $X$ lies in $N^{2 - \sigma}$ tubes of $\frak T$ with directions separated by angle $\gtrsim 1/N$.  This $X$ and $\frak T$ obeys Hypotheses \ref{uniform3transintro}.  For each cube $Q \in X$, the directions of the tubes of $\frak T$ through $Q$ lie within an angle $N^{-\sigma}$ of the $x_2 x_3$-plane.  This shows that the estimate in Proposition \ref{planypropintro} cannot be improved.  

In this example, the plane $\pi(Q)$ can be chosen to be the $x_2 x_3$-plane for all $Q \in X$.  Nevertheless, the size of the `grains' in this example is only $1 \times N^\sigma \times N^\sigma$.  If we take a segment $Seg \subset T \in \frak T$ of length $N^\sigma$, and if we take the $1 \times N^\sigma \times N^\sigma$ slab through $Seg$ parallel to the $x_2 x_3$-plane, then if any other tube $T' \in \frak T$ intersects $Seg$, then a segment of $T'$ of length $\sim N^{\sigma}$ lies in our slab.  In this sense, we have grains of size $1 \times N^\sigma \times N^\sigma$, and the grains cannot be made any larger.

In this example, the zero set $Z(P)$ could be a union of $N^{1 - \sigma}$ planes parallel to the $x_2 x_3$-plane, with $x_1$ coordinate equal to $1, 2, ..., N^{1 - \sigma}$.  We could also take a union of $\sim N^{1 - \sigma}$ planes that are not quite parallel to each other.  In any case, a typical tube of $\frak T$ hugs one plane for a length $\sim N^\sigma$, then shifts to another plane and hugs it and so on.  In our proof of Theorem \ref{graininessintro}, we will see that something like this picture occurs in general.  A typical tube of $\frak T$ hugs a nearly flat piece of $Z(P)$ for a length of $\sim N^\sigma$, then shifts to another nearly flat piece of $Z(P)$ and hugs it for a length $\sim N^\sigma$, and so on.  These nearly flat pieces of $Z(P)$ with diameter $\sim N^\sigma$ are the grains.

We consider another situation to show that we really need the tubes of $\frak T$ through a given $Q \in X$ to point in three different directions.  Consider the regulus defined by the equation $x_3 = x_1 x_2 / N$,  and then consider a neighborhood of this regulus given by $ | x_3 - x_1 x_2 /N| \le N^{1 - \sigma}$, $|x_1|, |x_2| \le N$.  Then we can let $X$ be the lattice unit cubes that intersect this neighborhood.  There are many tubes that intersect $\gtrsim N$ cubes of $X$.  Take any line in the regulus, thicken it to a tube, and then translate the tube vertically by a distance $\le N^{1 - \sigma}$.  We can choose $\frak T$ so that $X$ and $\frak T$ obey all of Hypotheses \ref{uniform3transintro}, except that the tubes through a given cube $Q \in X$ point in only two directions and not three directions.  In this case, $X$ and $\frak T$ have grains only at scale $N^{1/2}$.  If $\sigma > 1/2$, then the size of the grains is significantly smaller than $N^\sigma$.

In this second example, the degree reduction argument still applies.  The full degree reduction theorem, Theorem \ref{degredtube}, is more general than Theorem \ref{degredtubeintro}, and it applies to this example.  In this case, the surface $Z(P)$ could be $N^{1 - \sigma}$ parallel reguli.  But in this case, since there are only two tubes of $\frak T$ through a typical cube $X$, we are not able to get the same estimates for the curvature of $Z(P)$.

\subsection{Main ideas of the proof}

The proofs of our theorems are based on the arguments in \cite{GK}, but adapted to study tubes instead of lines.  We recall the outline of the proof of Theorem \ref{gk} from \cite{GK}, and we explain the main issues in adapting the proof to tubes.

To prove Theorem \ref{gk}, we consider a set $\frak L$ of $N^2$ lines in $\RR^3$, and a set $X$ with far fewer than $N^3$ points, where each line of $\frak L$ contains $N$ points of $X$.  We have to prove that many lines of $\frak L$ cluster in a plane.

The first step of the proof of Theorem \ref{gk} is a degree reduction argument.  We study the polynomials that vanish on the lines $\frak L$.  For any set of $N^2$ lines in $\RR^3$, there is a polynomial of degree $\sim N$ that vanishes on the lines.  But if $X$ is much smaller than $N^3$, then we can find a polynomial of much smaller degree.

The degree reduction involves two observations.  First, by a dimension counting argument, we can find a polynomial that vanishes at any $S$ points of $\RR^3$ with degree $\lesssim S^{1/3}$.  Once we have a polynomial that vanishes at some points, we can sometimes force it to vanish at other points by using the following simple vanishing lemma:

{\bf Vanishing Lemma.} If a polynomial $P$ vanishes at $ > \Deg P$ points on a line $l$, then it vanishes on the entire line.  

\noindent 

In particular, we let $P$ be a polynomial that vanishes on $X$ with degree $\lesssim |X|^{1/3}$, much smaller than $N$.  Since each line of $\frak L$ contains $N > \Deg P$ points of $X$, we see that $P$ vanishes on all the lines of $\frak L$.  We call this a contagious vanishing argument: the vanishing of $P$ spreads from the points of $X$ to the lines of $\frak L$.  In the paper below, we will use a more complicated contagious vanishing argument from \cite{GK} that gives a stronger estimate on the degree.  

Let's pause and discuss what happens when we replace lines by tubes and points by unit cubes.  We are immediately faced by a question: what does it mean for a polynomial to `vanish at a cube'.  If a polynomial is not identically zero, then it cannot vanish at every point of a cube.  The paper \cite{Gu1} suggested an approach to this issue.  We look for a polynomial that roughly bisects the cube, in the sense that $P > 0$ on roughly half the cube, and $P < 0$ on roughly half the cube.  If $P$ bisects a unit cube, then the area of $Z(P)$ in the cube is $\gtrsim 1$.  The generalized ham sandwich theorem \cite{ST} says that 
for any $S$ cubes in $\RR^3$, there is a polynomial that bisects all $S$ cubes with degree $\lesssim S^{1/3}$.

When we try to adapt the degree reduction argument to tubes, we need to generalize the vanishing lemma above to the context of cubes and tubes.  We may start with the following question: if a polynomial $P$ bisects $> \Deg P$ cubes along a tube, does it follow that $P$ (roughly) bisects all the cubes along the tube?  The answer is no.  This is a main source of difficulties in generalizing the arguments of \cite{D} from lines to tubes.  
For instance, consider the degree 10 plane curve $y = 10^{-100} x^{10}$.  For $|x| < 10^{9}$, this curve is very close to the $x$-axis, and it roughly bisects many unit squares along the $x$-axis.  But around $|x|=10^{10}$, the curve swerves sharply away from the $x$-axis, and it does not bisect any square of the $x$-axis farther out than this.  So the simplest generalization of the vanishing lemma to tubes fails.  But we will prove that a weaker statement still holds.  

To get a feel for this weaker statement, we first consider a simpler question in a similar spirit.  Let $\delta > 0$ be a small number, and suppose that a polynomial $P$ obeys $|P(x_j )| < \delta$ at $> \Deg P$ points $x_j$ along a line $l$.  Does it follow that $|P(x) | < \delta$ along the entire line $l$?  Again, the answer is easily seen to be no.  However, on the line $l$, the polynomial $P$ can take the value $\delta$ at most $\Deg P$ times, and it can take the value $- \delta$ at most $\Deg P$ times.  Therefore, if $|P(x_j)|  < \delta$ at $100 \Deg P$ points of a line $l$, then $|P(x)| < \delta$ on most of the line segments between these points.  

The vanishing lemma for tubes is in this spirit.  Roughly speaking, we will prove that if $P$ approximately bisects far more than $ \Deg P$ unit cubes along a tube, then $P$ approximately bisects the unit cubes in most of the `tube segments' between these cubes.  Lemma \ref{vanishtube} gives the precise statement.
This vanishing lemma for tubes is much weaker than the one for lines, but it is still strong enough to carry out the degree reduction argument, proving Theorem \ref{degredtubeintro}.

We now return to our outline of the proof of Theorem \ref{gk}.  We have found a polynomial $P$ that vanishes on the lines of $\frak L$ with good control of the degree.  Next we study its zero set: $Z(P)$.
An average point of $X$ lies in many lines of $\frak L$.  For the purposes of this discussion, we assume that each point of $X$ lies in at least three lines of $\frak L$.  Next we note that each point of $X$ must be a special point of the surface $Z(P)$.  If the lines of $\frak L$ through $x \in X$ are not coplanar, then $x$ must be a singular point of $Z(P)$.  If the lines of $\frak L$ through $x$ are coplanar, and if $x$ is a regular point of $Z(P)$, then $x$ must be a flat point of $Z(P)$ - a point where the second fundamental form of $Z(P)$ vanishes.

First we discuss singular points.  Singular points are contagious.  If a line $l \subset Z(P)$ contains more than $\Deg P$ singular points, then every point of $l$ is singular.   Moreover, $Z(P)$ can contain at most $\sim (\Deg P)^2$ singular lines, and $(\Deg P)^2$ is far less than $N^2$.  Therefore, most of the lines of $\frak L$ contain $< \Deg P$ singular points.  Since $\Deg P$ is far less than $N$, most lines of $\frak L$ contain nearly $N$ flat points.

This part of the argument generalizes to tubes using the methods of \cite{Gu1}.  It is closely related to planiness.  For most cubes $Q$ in $X$, Proposition \ref{planypropintro} says that most of the tubes of $\frak T$ passing through $Q$ lie within a small angle  of a certain plane $\pi(Q)$.  This says that for most cubes $Q$, the tubes through $Q$ are morally coplanar.  Here is an outline of the proof of Proposition \ref{planyprop}.
Fix a tube $T \in \frak T$.  We know that $X$ contains $\gtrsim N$ unit cubes that intersect $T$, and in each of these cubes $Z(P)$ has area $\gtrsim 1$.  Therefore, the area of $Z(P) \cap T$ is $\gtrsim N$.  But $\Deg P \lesssim N^{1 - \sigma}$, and so almost any line parallel to the center line of $T$ intersects $Z(P)$ at most $\Deg P \lesssim N^{1 - \sigma}$ times.  The only way that this can happen is for $v(T)$ to be nearly tangent to $Z(P)$ at most points of $Z(P) \cap T$.  Now consider a typical cube $Q$, which lies in several tubes.  Let $T_1$ and $T_2$ be cubes through $Q$ in quantitatively different directions.  Typically, at most points of $Z(P) \cap Q$, $T Z(P)$ makes a small angle with both $v(T_1)$ and $v(T_2)$.  There is a unique plane containing $v(T_1)$ and $v(T_2)$, and $TZ(P)$ is close to this plane at most points in $Q$.  This plane is $\pi(Q)$.  Most other tubes through $Q$ are nearly tangent to $TZ(P)$, and so they must be nearly tangent to $\pi(Q)$.  

We again return to the proof of Theorem \ref{gk}.  We have shown that most lines of $\frak L$ contain close to $N$ flat points of $Z(P)$.  We will use these flat points to force the lines of $\frak L$ to cluster into planes.  The following approach is based on \cite{EKS}.  At a flat point $x \in X$, all the lines of $\frak L$ through $x$ lie in a plane $\pi(x)$ which must be the tangent plane of $Z(P)$ at $x$.  Now \cite{GK} proves that flatness is also contagious: if $Z(P)$ is flat at $> 3 \Deg P$ points of a line $l$, then $Z(P)$ is flat at every point of $l$.  So it follows that $Z(P)$ is flat along most lines of $\frak L$.  Fix a line $l \subset Z(P)$ where $Z(P)$ is flat.  By elementary differential geometry, it follows that the tangent plane of $Z(P)$ is constant along $l$.  But then $\pi(x)$ is the same for all $x \in l$.  Call this plane $\pi(l)$.  Now we see that all the other lines that intersect $l$ (at flat points) must lie in the plane $\pi(l)$, and this causes clustering in planes.

The hardest part of this paper is to generalize this argument about flat points from lines to tubes.  Recall that if $x$ lies in three coplanar lines of $\frak L$, and if $Z(P)$ is non-singular at $x$, then $x$ is a flat point of $Z(P)$.  Does this basic result have an analogue for tubes?
Suppose we consider a cube $Q \in X$ lying in three tubes of $\frak T$, $T_1, T_2, T_3$, which are all nearly coplanar.  We can also assume that the angles between the three tubes are $\gtrsim 1$.  Recall that $Z(P)$ roughly bisects all the cubes of $X$, including many cubes in each of the tubes.  Does it follow that the curvature of $Z(P)$ is nearly zero in $Q$?  Perhaps surprisingly, the answer is morally yes.
We will prove that for most cubes $Q \in X$, for most points $x \in Z(P) \cap Q$, the second fundamental form of $Z(P)$ at $x$ has size $\lesssim N^{-2 \sigma}$.  We will also prove that this curvature bound is contagious, and that the second fundamental form is $\lesssim N^{-2 \sigma}$ at many points on the tube segments between the cubes.  This bound on the curvature controls how the tangent plane changes as we move along $Z(P)$, and so it bounds how the plane $\pi(Q)$ rotates as $Q$ moves along a tube $T$.  In particular, if we consider two cubes $Q, Q'$ along a tube $T$ with $\Dist(Q, Q') \lesssim N^\sigma$, then the angle between $\pi(Q), \pi(Q')$ is typically $\lesssim N^{-\sigma}$, proving Theorem \ref{graininessintro}.  

Let us sketch the proof of this curvature estimate.  Suppose for a moment that $Z(P)$ was just a graph of a degree 2 polynomial.  In other words, let's suppose that $Z(P)$ is defined by the equation $x_3 = A(x_1, x_2)$ where $A$ is a homogeneous degree 2 polynomial, and that $Q$ is centered at the origin and $\pi(Q)$ is the plane $x_3 = 0$.   The second fundamental form of $Z(P)$ at the origin is exactly $A$.  Suppose also that $Z(P)$ bisects cubes intersecting the three tubes $T_i$ out to a radius $R$.  Since the tubes $T_i$ make an angle $\lesssim N^{-\sigma}$ with $\pi(Q)$, we see that $A(x_1, x_2) \le R N^{-\sigma}$ for points $(x_1, x_2)$ at radius $R$ in the three tubes in three different directions.  This implies that the coefficents of $A$ are bounded by $N^{-\sigma} R^{-1}$.  How accurate is this model?  After all, $P$ is a high degree polynomial which makes $Z(P)$ a complicated surface.   The main work in the proof is to show that for a typical $Q, T_1, T_2, T_3$, the second fundamental form of $Z(P)$ is morally constant on the three tubes out to a radius $R \sim N^{\sigma}$.  In other words, for typical $Q, T_1, T_2, T_3$, the simple model above is an accurate model of $Z(P)$ on $T_1 \cup T_2 \cup T_3$ restricted to a ball of radius $\sim N^\sigma$.  This estimate depends on the degree bound $\Deg P \lesssim N^{1 - \sigma}$.  Its proof requires a mix of algebraic geometry and differential geometry.

\subsection{Organization of the paper}

In Section 1, we prove a version of the vanishing lemma for tubes.  In Section 2, we use this vanishing lemma to prove a degree reduction theorem for tubes.  This theorem implies Theorem \ref{degredtubeintro}, and it's a little stronger.  In Section 3, we review the Crofton formula - a result of integral geometry that gives bounds on the volumes of algebraic varieties.  We will use the Crofton formula repeatedly to control the geometry of $Z(P)$.  In Section 4, we use the degree reduction theorem to prove our  planiness and graininess estimates, Proposition \ref{planypropintro} and Theorem \ref{graininessintro}.

\section{Parameter counting and the vanishing lemma for tubes}

Suppose that $l$ is a line in $\RR^n$, and that $P$ is a polynomial of degree $\le D$ that vanishes at $> D$ points of $l$.  Then $P$ must vanish on all of $l$.  This basic result is sometimes called a vanishing lemma, and it plays a crucial role in polynomial method arguments about the intersection patterns of lines.

We want to formulate some analogue of this vanishing lemma when the line $l$ is replaced by a cylindrical tube $T$.  In this section we set up an analogy, and then state and prove a version of the vanishing lemma for tubes.  This version is a lot weaker than for lines, but it still has some applications.

Let $T \subset \RR^n$ be a cylindrical tube of radius 1 and arbitrary length.  We think of $T$ as analogous to a line $l$.  Let $Q$ denote a unit cube that intersects $T$.  We think of $Q$ as analogous to a point on $l$.

What does it mean for $P$ to ``vanish" on $Q$?  We build up to our definition in a few steps.  One possible definition was suggested in \cite{Gu1}.
Consider the sets $\{ x \in Q | P(x) > 0 \}$ and $\{ x \in Q | P(x) < 0 \}$.  We say that $P$ bisects $Q$ if 

$$ \Vol \{ x \in Q | P(x) > 0 \} = \Vol \{ x \in Q | P(x) < 0 \} = 1/2. $$

\noindent We could also relax this definition and say that $P$ ``vanishes" on $Q$ if each of these sets has volume at least $1/3$.  The parameter 1/3 is somewhat arbitrary, and we could adjust it.

We will need a definition that is a little stronger.  The stronger definition is somewhat analogous to saying that $P$ vanishes to high order at a point.

\begin{definition} Let $Q$ be a unit cube in $\RR^n$.  We say that a polynomial $P$ cuts $Q$ at scale $r$ if, for each ball $B$ of radius $\rho$ in the range $r \le \rho \le 1$ and with center at distance $\le 1/r$ from $Q$, we have

$$ (\frac{1}{2} - r) \Vol B \le \Vol \{ x \in B | P(x) > 0 \} \le (\frac{1}{2} + r) \Vol B. $$

\end{definition}

The definition has a little to digest.  One main point is that as $r$ gets smaller, the definition gets stronger.  In a sense, there are really three parameters here: the radii of the balls, the distance to $Q$, and the error-tolerance in the near-bisection inequality.  But it's easier to just keep track of one parameter $r$, and we don't lose anything in the arguments in the paper below.  As a rough analogy, $P$ cuts $Q$ at scale $r$ is like saying $P$ vanishes at a point $q$ to order $r^{-n}$.  This stronger condition is more contagious than the simpler condition we started with above.

In the polynomial method, it is important to be able to find polynomials that vanish at given points.  For ordinary vanishing, the most fundamental result of this type is the following parameter-counting lemma.  Let $\Poly_D (\RR^n)$ denote the vector space of polynomials on $\RR^n$ of degree $\le D$.

\begin{lemma} \label{parcount} (Parameter counting) Let $q_1, ..., q_S$ be a set of points in $\RR^n$, and suppose that $S < \Dim \Poly_D(\RR^n)$.  Then there is a non-zero polynomial of degree $\le D$ that vanishes at all the points $q_i$.  
\end{lemma}

\begin{proof} Consider the linear map $E: \Poly_D (\RR^n) \rightarrow \RR^S$ defined by $E(P) = (P(q_1), ..., P(q_S))$.  By hypothesis, the dimension of the domain is larger than the dimension of the range, so the linear map $E$ has a non-trivial kernel.  A non-zero element of this kernel is a non-zero polynomial of degree $\le D$ that vanishes at all the points $q_i$.
\end{proof}

We remark that the dimension of $\Poly_D (\RR^n)$ is ${D + n \choose n} \ge D^n / n! \ge D^n / n^n$.  Therefore, for any set of $S$ points in $\RR^n$, we can find a non-zero polynomial vanishing on these points with degree $\le n |S|^{1/n}$.

This lemma has a good analogue for our cutting definition.

\begin{lemma} \label{parcountcut} (Parameter counting for cubes) There is a small constant $c$ and a large power $a$ depending only on the dimension $n$ so that the following holds.  Let $r \in (0,1/2)$ be a real number.  Let $Q_1, ... Q_S$ be a set of unit cubes in $\RR^n$, and suppose that $ S < c r^a \Dim \Poly_D(\RR^n)$.  Then there is a non-zero polynomial of degree $\le D$ that cuts each cube $Q_i$ at scale $r$.  
\end{lemma}

This lemma follows from the polynomial ham sandwich theorem.

\begin{theorem} (Polynomial ham sandwich theorem, Stone and Tukey, \cite{ST}, see also \cite{Gu1}) Suppose that $U_1, ..., U_S$ are finite volume open sets in $\RR^n$ and that $S < \Dim \Poly_D(\RR^n)$.  Then there is a non-zero polynomial of degree $\le D$ that bisects each $U_i$.
\end{theorem}

Now we can give the proof of Lemma \ref{parcountcut}.

\begin{proof} Consider a lattice of cubes of side length $(10n)^{-2n} r^2$.  Notice that the diameter of such a cube is $\le (10 n)^{-n} r^2$.  Let $\{ U_i \}$ be the set of cubes in the lattice which intersect the $(10/r)$-neighborhood of the union of the cubes $Q_i$.  The number of such $U_i$ is $\le C(n) r^{-a(n)} S$, where $S$ is the number of cubes $Q_i$.  By hypothesis, $C(n) r^{-a(n)} S < \Dim \Poly_D(\RR^n)$, so we can choose a polynomial of degree $\le D$ that bisects each cube $U_i$.

Now we consider a ball $B$ with radius in the range $(r,1)$ and center within distance $1/r$ of one of the cubes $Q_i$.  We can write $B$ as a union of some of the small cubes $U_i$ plus a small leftover piece.  The leftover piece is contained in the $(10 n)^{-n} r^2$ neighborhood of the boundary of $B$.  Now an elementary computation shows that the volume of the leftover piece is $\le (1/10) r \Vol B$.  
The polynomial $P$ exactly bisects each small cube $U_i$, and so it obeys the desired inequality for $B$.  \end{proof}

\begin{cor} \label{parcountcut'} For each $n$, there are constants $C_n, a_n$ so that the following holds.  Let $Q_1, ..., Q_S$ be unit cubes in $\RR^n$, and let $r \in (0, 1/2)$ be given.  Then there is a polynomial $P$ that cuts each cube $Q_i$ at scale $r$ with $\Deg P \le C_n r^{-a_n} S^{1/n}$.
\end{cor}

\begin{proof} By Lemma \ref{parcountcut}, we can find a polynomial $P$ of degree $\le D$ cutting all the cubes at scale $r$ as long as 

$$S < c_n r^{a_n} \Dim \Poly_D(\RR^n).$$  

Now $\Dim \Poly_D(\RR^n) = { D + n \choose n} \ge D^n / n!$, so it suffices to check

$$ S < c_n (n!)^{-1} r^{a_n} D^n. $$

We can find an integer $D$ obeying this inequality with $D \le C_n r^{-a_n} S^{1/n}$.  \end{proof}

Now we turn to the analogue of the vanishing lemma. Suppose that $T$ is a tube of radius 1 and that $\{ Q_i \}$ are some unit cubes that intersect $T$.  Also suppose that the distance between any two $Q_i$ is $\ge 2n$.  Because of the separation, the cubes $Q_i$ have a definite order along $T$.  They divide $T$ into segments between the $Q_i$.  We can make this precise in the following way.
After rotation and translation, we can arrange that $T$ is described in coordinates $x_1, ..., x_n$ by the inequalities $\sum_{j=1}^{n-1} x_j^2 \le 1$ and $x_n \in [h_s,h_f]$.  We let $h_1, ..., h_S$ be the $x_n$ coordinates of the centers of the cubes $Q_i$, and we renumber the $Q_i$ so that $h_1 < h_2 < ... < h_S$.  Since the distances between the cubes are $\ge 2n$ and they all intersect $T$, it's straightforward to check that the gaps $h_i - h_{i-1}$ are all $\ge 1$.  Now we divide $T$ into tube segments $T_i$ defined by $\sum_{j=1}^{n-1} x_j^2 \le 1$ and $x_n \in [h_i, h_{i+1}]$.  

Our vanishing lemma roughly says that if $P$ is a polynomial of degree $\le D$, and if $P$ cuts far more than $D$ unit cubes $Q_i$ that intersect $T$, then $P$ cuts the unit cubes touching most of the tube segments between them.

\begin{lemma} \label{vanishtube} (Vanishing lemma for tubes) For each dimension $n$, there is a small $r(n) > 0 $ and a large constant $C(n)$ so that the following holds.  Suppose that $P$ is a non-zero polynomial in $\Poly_D(\RR^n)$.  Suppose that $Q_1, ..., Q_S$ are unit cubes that intersect $T$ with pairwise distance $\ge 2n$, and suppose that $P$ cuts each $Q_i$ at scale $r \le r(n)$.  Let $T_i$ be the tube segments of $T$ defined above.  There are $\le C(n) r^{-4n} D$ bad tube segments, and the rest of the $T_i$ are good tube segments.  If $Q$ is a unit cube that intersects a good tube segment, then $P$ cuts $Q$ at scale $2 r$.
\end{lemma}

The key difference between the vanishing lemma for tubes and for lines is that in the case of lines there were no bad segments.  If $P$ vanishes at $> D$ points $x_i$ along a line $l$, then it must vanish on the whole line, including points far away from the $x_i$.  Our lemma for tubes does not say anything about what happens along the tube far beyond all the cubes $Q_i$.  It only describes what happens between the cubes and there can be $\sim D$ bad tube segments where we are unable to say anything.  On the other hand, if $P$ cuts many times $D$ evenly spaced cubes $Q_i$ along $T$, then $P$ must cut most of the cubes between them.

\begin{proof} As above, we choose coordinates so that $T$ is defined by $\sum_{j=1}^{n-1} x_j^2 \le 1$.  We let $\pi$ denote the projection onto the first $(n-1)$ coordinates: $\pi(x_1, ..., x_n) = (x_1, ..., x_{n-1})$.  For any $y \in \RR^{n-1}$, we call the line $\pi^{-1}(y)$ a vertical line.  If a vertical line is not contained in $Z(P)$, then it intersects $Z(P)$ in $\le D$ points.  Also, the set of $y$ so that $\pi^{-1}(y) \subset Z(P)$ has measure zero.  So for almost every $y$, $\pi^{-1}(y)$ intersects $Z(P)$ in $\le D$ of the tube segments $T_i$.  Also, $\pi T_i \subset \pi T$ which is a unit ball.  Therefore, we get the following estimate:

$$ \sum_i \Vol_{n-1} \pi (Z(P) \cap T_i) \le C(n) D. $$

We need a small variation of this inequality involving the $R$-neighborhood of $T_i$, written $N_R T_i$.  The $N_R T_i$ are not disjoint.  However, the consecutive heights differ by at least 1: $h_i - h_{i-1} \ge 1$.  So any point lies in $\le 2R + 1$ of the sets $N_R T_i$.  Also $\pi N_R T_i \subset \pi N_R T$, which is a ball of radius $R + 1$.  Therefore, we get the following estimate:

$$  \sum_i \Vol_{n-1} \pi (Z(P) \cap N_R T_i) \le C(n) D (R+1)^{n}. \eqno{(1)}$$

From now on, we take $R = n + (1/r)$, so that all the balls and cubes in our story lie in the $R$-neighborhood of $T$.  

We call $T_i$ good if $\Vol_{n-1} \pi (Z(P) \cap N_R T_i) \le (100n)^{-n} r^{2n}$.  Otherwise, we call $T_i$ bad.  We see from equation (1) that the number of bad $T_i$ is $\le C(n) r^{-4n} D$ as desired.

Now let $Q$ be a unit cube that intersects a good segment $T_i$.  We have to prove that $P$ cuts $Q$ at scale $2 r$.  Let $B$ be a ball with radius in the range $[2 r, 1]$, and with center a distance $\le (1/2) r^{-1}$ from $Q$.  We have to prove that $P$ nearly bisects $B$.

The tube segment $T_i$ runs in the range $x_n \in [h_i, h_{i+1}]$, where $h_i$ is the $x_n$-coordinate of $Q_i$.  We consider a translate of $B$ in the $x_n$ direction.  We let $B'$ be the 
translation with center at height $h_i$.  

If $r(n)$ is sufficiently small, then the ball $B'$ lies in the $r^{-1}$ neighborhood of $Q_i$.  Therefore, $P$ nearly bisects $B'$:

$$ |\Vol \{ x \in B' | P(x) > 0 \} - (1/2) \Vol B'|  \le r \Vol B'.  \eqno{(2)}$$

The main idea is that $P$ cuts $B$ and $B'$ in a similar way, because $\Vol_{n-1} \pi (Z(P) \cap N_R T_i)$ is very small.  Note that $B$ and $B'$ are both in $N_R T_i$.  

We call a line $l$ vertical if it is parallel to the $x_n$ axis.  We say a vertical line is empty if $l \cap Z(P) \cap N_R T_i$ is empty.  If $l$ is an empty line, then the sign of $P$ does not change on $l \cap N_R T_i$.  Also, if $l$ is any vertical line, then the length of $l \cap B$ is the same as the length of $l \cap B'$.  We let $E$ be the union of all the empty lines.  By the above discussion, we see the following:

$$ \Vol \{ x \in B \cap E | P(x) > 0 \} = \Vol \{ x \in B' \cap E | P(x) > 0 \}. $$

On the other hand, $B \cap E^c$ and $B' \cap E^c$ are extremely small.  Since the radius of $B$ is $\le 1$, $\Vol_n (B \cap E^c) \le  2 \Vol_{n-1} \pi (E^c \cap B) \le 2 \Vol_{n-1} \pi (Z(P) \cap N_R T_i) \le  2 (100n)^{-n} r^{2n}. $  In particular, 
$\Vol_n E^c \cap B \le (1/100) r \Vol B$.  By the same argument, $\Vol_n E^c \cap B' \le (1/100) r \Vol B = (1/100) r \Vol B'$.  Therefore,

$$ |\Vol \{ x \in B | P(x) > 0 \} - \Vol \{ x \in B' | P(x) > 0 \}| \le (2/100) r \Vol B. \eqno{(3)}$$

Combining inequalities $(2)$ and $(3)$, we see

$$ |\Vol \{ x \in B | P(x) > 0 \} - (1/2) \Vol B|  \le (1.02) r \Vol B.  $$

This proves the desired near-bisection inequality for the ball $B$.
\end{proof}

\section{Degree reduction for tubes}

In this section, we use parameter counting and the vanishing lemma to prove a version of degree reduction for tubes.  To orient ourselves, we first present a parallel version of degree reduction for lines, and recall the proof.

\subsection{Degree reduction for lines}

The following Proposition is a degree reduction result for lines in $\RR^3$.  The statement and the proof are models for the result we will prove for tubes.  The proof here works over any field, so we present it in that generality.

\begin{prop} Let $\FF$ be a field.  Let $\epsilon > 0$ and $E > 0$ be any numbers.  
Suppose that $\frak L$ is a set of lines in $\FF^3$ and $X$ is a set of points in $\FF^3$ obeying the following conditions.

\begin{enumerate}

\item Each line $l \in \frak L$ contains between $N$ and $E N$ points of $X$, for some number $N$.

\item Each point of $X$ lies in between $\rho$ and $E \rho$ lines of $\frak L$, for some $\rho \ge 2$.

\end{enumerate}

Then there is a non-zero polynomial $P$ of degree $\le \Poly(E, \epsilon^{-1}) |X| N^{-2}$ that vanishes on $\ge (1 - \epsilon) |X|$ points of $X$.  

\end{prop}

By the parameter counting lemma, Lemma \ref{parcount}, there is a non-zero polynomial vanishing on $X$ with degree $\le C |X|^{1/3}$.  If $|X|$ is much less than $N^3$, then $|X| N^{-2} < |X|^{1/3}$, and we get a significantly lower degree.  Therefore, we call this type of estimate a degree reduction result.

The estimate is particularly sharp over finite fields.  Suppose that $\FF = \FF_q$ is the finite field with $q$ elements, and suppose that $N = q$.  The Proposition tells us that there is a polynomial $P$ vanishing on most of $X$ with $\Deg P \lesssim |X| q^{-2}$.  On the other hand, the Schwarz-Zippel lemma says that a polynomial $P$ vanishes on at most $(\Deg P) q^2$ points.  Therefore, $\Deg P \gtrsim |X| q^{-2}$, and so the degree estimate is sharp up to a constant factor.

\begin{proof} Here is an outline of the proof.  Later when we do degree reduction for tubes, we will follow the same outline.

Step 1. We pick a random subset of lines $\frak L_1 \subset \frak L$.  We pick a bunch of points on each line $l \in \frak L_1$.  Then we use parameter counting to find a polynomial $P$ of controlled degree that vanishes at all the points.

Step 2. By the vanishing lemma, $P$ vanishes on each line of $\frak L_1$.

Step 3. Since there are many intersections, we will prove that each line of $\frak L$ usually has many intersection points with lines of $\frak L_1$.  We know that $P$ vanishes at each of these intersection points.

Step 4. By the vanishing lemma, $P$ vanishes on most lines of $\frak L$.  Therefore, it vanishes at most points of $X$.

Now we begin the details.  We let $D$ denote the degree bound for $P$.  We take $D = K |X| N^{-2}$, where $K = C (E \epsilon^{-1} )^A$ for some large constants $C, A$.  

(We want $\Deg P \le D$ to be an integer, so we need to check that this $D \ge 1$.  It suffices to check that $|X| \ge N^2/2$.  Let $l$ be a line of $\frak L$.  We know $l$ contains $\ge N$ points of $X$.  Each of these points lies in another line of $\frak L$, so we know that $\frak L$ contains at least $N$ lines besides $l$: let's call them $l_1, l_2, ...$ Now $l_1$ contains $\ge N-1$ points of $X$ not in $l$.  And more generally, $l_i$ contains at least $N-i$ points of $X$ not in $l, l_1, ..., l_{i-1}$.  So the total number of points of $X$ is at least $N + (N-1) + ... + 1 \ge (1/2) N^2$. )

Step 1. We randomly pick a set $\frak L_1 \subset \frak L$ by including each line with probability $(1/100) D^2 |\frak L|^{-1}$.  With high probability, the number of lines in $\frak L$ is $\le (1/10) D^2$.  

We pick $2 D $ points on each line of $\frak L_1$.  The total number of points picked is $\le (1/5) D^3$.  By the parameter counting lemma, Lemma \ref{parcount}, there is a non-zero polynomial $P$ which vanishes on these points and has $\Deg P \le D$.

Step 2. On each line of $\frak L_1$, the polynomial $P$ vanishes at $2 D > \Deg P$ points.  By the vanishing lemma, $P$ vanishes on each line of $\frak L_1$.

Step 3. Next we want to prove that with high probability, $P$ vanishes at many points on most lines of $\frak L$.
Let $l'$ be a fixed line of $\frak L$.  We first estimate the expected number of points of $l'$ that lie in a line of $\frak L_1$.  

The lines of $\frak L_1$ contain $\sim N D^2$ points of $X$.  The probability that a point $x \in X$ lies in a line of $\frak L_1$ is constant on $X$ up to a factor $\Poly(E)$.  So the probability that $x$ lies in a line of $\frak L_1$ is $\ge \Poly(E)^{-1} N D^2 |X|^{-1}$.  The line
$l'$ contains $\ge N $ points of $X$.  Therefore, the expected number of points of $l'$ in the lines of $\frak L_1$ is
$\ge \Poly(E)^{-1} N^2 D^2 |X|^{-1} = \Poly(E)^{-1} K D$.  By choosing the exponent $A$ large enough in the definition of $K$, we can arrange that this expected number is $> 20 D$.

Now we would like to prove that with high probability, the line $l$ contains $> D$ intersection points with lines of $\frak L_1$.  Let $x_1, x_2, ..., x_N$ be points of $X \cap l'$.  Let $I(x_i)$ denote the event that $x_i$ lies in a line of $\frak L_1$ other than $l'$.  This definition is good because the events $I(x_i)$ are independent.  As we saw in the last paragraph, each event $I(x_i)$ occurs with probability $\ge \Poly(E)^{-1} N D^2 |X|^{-1}$.  If we choose $A$ large enough, the expected number of $I(x_i)$ that occur is $> 20 D$.  In fact we can do a little better and say that the expected number of $I(x_i)$ that occur is $> 20 D \epsilon^{-10} E^{10}$.  Since the $I(x_i)$ are independent, we have that $> 2 D$ of the events $I(x_i)$ occur with high probability.  So with probability $(1 - \epsilon^8 E^{-8})$, the line $l'$ contains $> 2 D$ intersection points with lines of $\frak L_1$.  

Now we can choose a particular $P$ so that at least $(1 - \epsilon^8 E^{-8}) |\frak L|$ lines of $\frak L$ contain at least $2D$ points where $P$ vanishes.  

Step 4. By the vanishing lemma, $P$ vanishes on at least $(1 - \epsilon^8 E^{-8}) |\frak L|$ lines of $\frak L$.  
Since each point of $X$ lies in approximately the same number of lines, it follows that $P$ vanishes on $(1 - \epsilon) |X|$ points of $X$.  
\end{proof}

\subsection{Degree reduction for tubes}

We now formulate a similar degree reduction result for tubes.  Instead of $P$ vanishing at a point, we discuss $P$ cutting a cube at a small scale $r$.  Also, using tubes, we need to pay attention to angles of intersection, and we add an extra transversality assumption.

\begin{theorem} \label{degredtube} Let $\epsilon > 0$ and let $E > 1$.  
Suppose that $\frak T$ is a set of tubes in $\RR^3$ with radius 1 and arbitrary length.  Suppose that $X$ is a set of disjoint unit cubes in $\RR^3$.  Suppose that $X$ and $\frak T$ obey the following conditions:

\begin{enumerate}

\item Each tube $T \in \frak T$ intersects between $N$ and $E N$ cubes of $X$, for some number $N$.

\item Each cube of $X$ intersects between $\rho$ and $E \rho$ tubes of $\frak T$, for some $\rho \ge 2$.

\item (transversality) For each cube $Q \in X$, and for each unit vector $v \in \RR^3$, a fraction $E^{-1}$ of the tubes of $\frak T$ that intersect $Q$ have angle $\ge E^{-1}$ with the vector $v$.

\end{enumerate}

Then there is a non-zero polynomial $P$ of degree $\le \Poly(E, \epsilon^{-1}) |X| N^{-2}$ that cuts $\ge (1 - \epsilon) |X|$ cubes of $|X|$ at scale $\epsilon$.  
\end{theorem}

The degree estimate in this theorem is sharp up to a constant factor when $X$ is contained in a ball of radius $\sim N$.  
If $X$ is contained in a ball of radius $\sim N$ and $P$ cuts most unit cubes of $X$, then the area of $Z(P) \cap B(N)$ is $\gtrsim |X|$.  On the other hand, the Crofton formula implies that for any polynomial $P$, the area of $Z(P) \cap B(N)$ is $\lesssim (\Deg P) N^2$.  (We will review the Crofton formula in Section \ref{revintgeom}.)  Comparing these inequalities, we see $\Deg P \lesssim |X| N^{-2}$. 
This situation is analogous to the finite field situation we discussed after the degree reduction proposition for lines.

\begin{proof}  We begin by making an outline of the proof, parallel to the case of lines.

Step 1. We pick a random subset of tubes $\frak T_1 \subset \frak T$.  We pick a bunch of cubes on each tube $T \in \frak T_1$.  Then we use parameter counting to find a polynomial $P$ of controlled degree that cuts all of the cubes.

Step 2. We apply the vanishing lemma for tubes to each tube $T \in \frak T_1$.   By Step 1, we know that $P$ cuts many cubes on $T$.  These cubes divide $T$ into a sequence of tube segments, and the vanishing lemma says that $P$ cuts the cubes in most of these segments.  We call the segments where $P$ cuts good segments.

Step 3. Let $T'$ be a typical tube of $\frak T$.  By assumption $T'$ has many intersections with other tubes of $\frak T$, and so $T'$ usually has many intersections with tubes of $\frak T_1$.  Being a little more careful, we will show that $T'$ usually intersects many tubes of $\frak T_1$ in good segments.  If $T'$ intersects a tube of $\frak T_1$ in a good segment, then we call the cube where they intersect a good cube.  By Step 2, we know that $P$ cuts every good cube.  

Step 4. By the vanishing lemma, $P$ cuts the cubes of $\frak T'$ in most of the segments between the good cubes from Step 3.  We next have to check that these good cubes are usually evenly distributed along $\frak T'$.  Then it follows that $P$ cuts most of the cubes in $\frak T'$.  Since this holds for most tubes $\frak T'$, $P$ cuts most cubes of $X$.

As before, we define $D = K |X| N^{-2}$, where $K = C (E \epsilon^{-1})^A$ for large constants $C, A$.   By the same argument as above, we can check that $|X| \ge \Poly(E)^{-1} N^2$, and so $D \ge 1$.  

We write $K^+$ for a small positive power of $K$, and $K^-$ for a small negative power of $K$.  These powers can change from line to line.  By choosing $A$ large, any $K^+$ is always at least $(E \epsilon^{-1})^{10}$.  On the other hand the power $K^+$ is always $\le K^{1/100}$, so that an expression like $K^{-(1/2)+}$ is much smaller than 1.

Step 1.  Now we choose a random subset of tubes $\frak T_1 \subset \frak T$ and some cubes on each tube.  When we worked with lines, we chose $2 D$ points on each line, which is enough to apply the vanishing lemma for lines.  But the vanishing lemma for tubes works better if we have far more than $D$ cubes in a tube.  
So we choose $ \gg D$ cubes on each tube, and we have to choose fewer tubes.  Playing around with the parameters it turns out to work if we take around $K^{1/2} D$ cubes on each tube, and around $K^{-1/2} D^2$ tubes.  All we really need about $K^{1/2}$ is that $K \gg K^{1/2} \gg 1$.  

Let $\frak T_1 \subset \frak T$ be a random subset of tubes, where each tube is selected with probability $K^{-1/2} D^2 |\frak T|^{-1}$.  With high probability, the size of $\frak T_1$ is $\lesssim K^{-(1/2)+} D^2$.  

For each tube $T \in \frak T_1$, we choose $K^{1/2} D$ cubes of $X$ which intersect $T$.  We choose them evenly spaced among the cubes of $X$ that meet $T$.  

Since we choose $K^{1/2} D$ cubes in each tube of $\frak T_1$, the total number of chosen cubes is $\le K^+ D^3$ with high probability.  By the parameter counting lemma for cutting cubes, Corollary \ref{parcountcut'}, we can find a non-zero $P$ with $\Deg P \le K^+  D$ which cuts every chosen cube at scale $K^-$.  

Step 2. Next we apply the vanishing lemma for tubes, Lemma \ref{vanishtube}, to each tube $T \in \frak T_1$.  
Fix a tube $T \in \frak T_1$, and let $Q_1, Q_2, ..., Q_{K^{1/2} D}$ be the chosen cubes that intersect $T$.  We label them in order.  Let $T_{i, i+1}$ be the segment of $T$ from $Q_i$ to $Q_{i+1}$, as defined before the statement of Lemma \ref{vanishtube}.  Note that there are $K^{1/2} D$ of these tube segments, and each of them intersects $\sim N K^{-1/2} D^{-1}$ cubes of $X$.
Lemma \ref{vanishtube} says that there are at most $K^+ D$ bad tube segments $T_{i, i+1}$, and that for any unit cube $Q$ intersecting any good tube segment, $P$ cuts $Q$ at scale $K^-$.  

Step 3. Let $Q$ be a fixed cube of $X$, and let $I(Q)$ be the event that $Q$ intersects a tube of $\frak T_1$.  We claim that the probability of $I(Q)$ is $\ge K^{(1/2)-} D N^{-1}$.  Each tube belongs to $\frak T_1$ with probability $K^{-1/2} D^2 |\frak T|^{-1}$.  There are $\ge \rho$ tubes of $\frak T$ that intersect $Q$, and so the probability of $I(Q)$ is at least $\rho K^{-1/2} D^2 |\frak T|^{-1}$.
We can simplify this expression using a double counting argument.  We count the incidences between tubes of $\frak T$ and cubes of $X$ in two different ways.  Up to powers of $E$, the number of incidences is $\rho |X|$ and it is also $N |\frak T|$.  Therefore, $\rho \ge K^- N |\frak T| |X|^{-1}$.  Plugging this in above, we see that the probability of $I(Q)$ is at least $K^{-(1/2)-} N D^2 |X|^{-1}$.  Plugging in $D = K |X| N^{-2}$, the probability of $I(Q)$ is at least $K^{(1/2)-} D N^{-1}$.  This proves the claim.

Now let $T'$ be an arbitrary tube of $\frak T$.  Let $Q_1, ..., Q_N$ be cubes of $X$ that intersect $T'$.  (These are different from the tubes in Step 2.)  The expected number of cubes $Q_j$ so that $I(Q_j)$ holds
is $\ge K^{(1/2)-} D$.  However, the events $I(Q_j)$ are not independent.  The problem is that a tube $T \in \frak T$ with a small angle to $T'$ may intersect many cubes $Q_j$, and if this tube $T$ is chosen for $\frak T_1$, it will cause $I(Q_j)$ to happen for many $j$.  

We can fix this independence problem by tweaking the definition, and considering only transverse intersections.  Among the cubes of $X$ that intersect $T'$, choose $E^{-1} N$ evenly spaced cubes $Q'_j$.  Between $Q_j'$ and $Q_{j+1}'$ there are $E$ cubes of $X$ that intersect $T'$, and so the distance from $Q'_j$ to $Q'_{j+1}$ is at least $E$.  We let $I_{tr}(Q_j')$ be the event that a cube of $\frak T_1$ intersects $Q_j'$ and the angle between that tube and $T'$ is at least $E^{-1}$.  If the angle between $T$ and $T'$ is at least $E^{-1}$, then $T$ can intersect at most one of the cubes $Q_j'$.  Therefore, the events $I_{tr}(Q_j')$ are independent.  

The transversality hypothesis in the statement of the Theorem says that for each cube $Q_j' \in X$, among the tubes of $\frak T$ that intersect $Q_j'$, at least a fraction $E^{-1}$ of them are $E^{-1}$-transverse to $T'$.  So by the same analysis as above, the probability of $I_{tr}(Q_j')$ is still $\ge K^{(1/2)-} D N^{-1}$.  The number of cubes $Q_j'$ is $E^- N$.  Therefore, the expected number of $Q_j'$ for which $I_{tr}(Q_j')$ occurs is $\ge K^{(1/2)-} D$.  
Moreover, since the events $I_{tr}(Q_j')$ are independent, we can say that with probability $(1 - K^-)$, there are $\ge K^{(1/2)-} D$ cubes $Q_j'$ where $I_{tr}(Q_j')$ holds.  

Suppose that $I_{tr}(Q_j')$ occurs.  It would be nice if we could conclude that $P$ cuts $Q_j'$ at scale $K^-$.  However we don't know this.  Since $I_{tr}(Q_j')$ occurs, we know that $Q_j'$ intersects a tube $T \in \frak T_1$, but $Q_j'$ may intersect a bad tube segment of $T$.  We would like to prove that this is rare.  I don't know how to prove this for a single $T'$, so we now have to average over all $T' \in \frak T$.

Let's make a little more notation.  For each tube $T' \in \frak T$, let $X_{spaced}(T') \subset X$ be a set of $E^{-1} N$ evenly spaced cubes among the cubes of $X$ that intersect $T'$.  The event $I_{tr}(Q_j')$ really depends on $T'$, and we make this explicit by calling it $I_{tr}(T', Q_j')$.  Let us formally state what we proved so far in our new notation.  

\begin{lemma} \label{transinters} For each $T' \in \frak T$, with probability $(1 - K^-)$, there are $\ge K^{(1/2)-} D$ cubes $Q_j'$ in $X_{spaced}(T')$ so that $I_{tr}(T', Q_j')$ holds.
\end{lemma}

If $Q_j' \in X_{spaced}(T')$, we let $I_{bad}(T', Q_j')$ be the event that $Q_j'$ lies in a bad segment of a tube $T \in \frak T_1$ and the angle between $T$ and $T'$ is $\ge E^{-1}$.  We will prove the following bound showing that $I_{bad}$ is rare.

\begin{lemma} \label{badinters} With probability $(1- K^-)$, 

$$ \Avg_{T' \in \frak T} \left| \{ Q_j' \in X_{spaced}(T') \textrm{ so that } I_{bad}(T', Q_j') \} \right| \le K^{+} D. $$
\end{lemma}

\begin{proof} Each tube $T \in \frak T_1$ intersects $\le K^+ N$ cubes of $X$.  The tube $T$ is divided into $K^{1/2} D$ segments, each containing the same number of cubes, and there are only $K^+ D$ bad segments.  Therefore, the number of cubes of $T$ in bad segments is $\le K^{-(1/2)+} N$.  With probablity $(1 - K^-)$, there are at most $K^{-(1/2)+} D^2$ tubes in $\frak T_1$, and so the total number of cubes in the bad segments of all these tubes is
$\le K^{-1+} D^2 N$.  Each of these cubes lies in at most $K^+ \rho \le K^+ |\frak T| N |X|^{-1}$ tubes of $\frak T$.  Therefore, the total number of bad events $I_{bad}(T', Q_j')$ is at most $K^{-1+} D^2 N^2 |X|^{-1} |\frak T| = K^+ D |\frak T|$.  Averaging over $T' \in \frak T$, we get the inequality above. \end{proof}

We note that if $I_{tr}(T', Q_j')$ holds but $I_{bad}(T', Q_j')$ does not hold, then $Q$ must intersect a tube $T \in \frak T_1$ in a good tube segment, and so $P$ cuts $Q_j'$ at scale $K^-$.  Comparing Lemma \ref{transinters} and Lemma \ref{badinters}, we see that with probability $(1 - K^-)$, 
for at least $(1 - K^-) |\frak T|$ tubes $T' \in \frak T$, $P$ cuts at least $ K^{(1/2)-} D$ cubes $Q_j' \in 
X_{space}(T')$ at scale $K^-$.

In Step 4, we will apply the vanishing lemma for tubes to $T'$.  For a typical $T'$, we see that $P$ cuts at least $K^{(1/2)-} D$ cubes along $T'$.  The vanishing lemma implies that $P$ also cuts the cubes in most of the segments between these cubes.  But to get a good estimate, we will need to know that these $K^{(1/2)-} D$ cubes are fairly evenly distributed along $T'$.  

Let us make this precise.  Consider a tube $T'$.  As $Q_j'$ varies in $X_{spaced}(T')$, the events $I_{tr}(T', Q_j')$ are independent.  Therefore the cubes $Q_j'$ where $I_{tr}(T', Q_j')$ holds are usually distributed very evenly.  More precisely, with probability $(1 - K^-)$, any $D$ tube segments between the cubes $Q_j' \in X_{spaced}(T')$ where $I_{tr}(T', Q_j')$ holds will intersect $\le K^{-(1/2)+} N$ cubes of $X$.

This holds for the following reason.  Let the cubes $Q_j' \in X_{spaced}(T')$ where $I_{tr}(T', Q_j')$ holds be called transverse intersection cubes.  We want to understand the tube segments between the transverse intersection cubes.  We define the `length' of a tube segment to be the number of cubes of $X$ that it intersects.  Define 

$$\lambda:= K^{-(1/2)+} D^{-1} N.$$ 

\noindent $\lambda$ is the typical length of a tube segment.  Now let $\beta > 1$ be a parameter, and consider a sequence of $\beta \lambda$ consecutive cubes in $X_{spaced}(T')$.  We consider the probability that these cubes lie in a single tube segment  - this is the same as the probability that none of the cubes in the sequence is a transverse intersection cube.  Since $I_{tr}(T', Q_j')$ holds with probability at least $K^{(1/2)-} D N^{-1} = \lambda^{-1}$, the probability that our sequence lies in a single tube segment is $\le e^{- \beta}$.  Next, divide the cubes of $X_{spaced}(T')$ into disjoint sequences of $\beta \lambda$ consecutive cubes.  There are $\le K^+ N \beta^{-1} \lambda^{-1}$ of these sequences.  Any tube segment of length $\ge 2 E \beta \lambda$ must contain one of these sequences.  Therefore, the expected number of such tube segments is bounded as follows:  

$$ \mathbb{E} \left[ \textrm{The number of tube segments of length } \ge 2 E \beta \lambda \right] \le e^{-\beta} K^+ N \beta^{-1} \lambda^{-1}. $$

\noindent This is the key formula in the proof.  In particular, it follows that with probability 
$(1 - K^-)$, 

$$[\textrm{The total length of all tube segments of length } \ge 2 E (\log K) \lambda] \le K^{-1} K^+ N. $$

\noindent On the other hand, any $D$ tube segments with length $\le 2 E (\log K) \lambda$ have total length $\le D K^+ \lambda \le K^{-(1/2)+} N$.  Therefore, with probability $(1 - K^-)$, the total length of any $D$ tube segments is $\le K^{-(1/2)+} N$.  

Here is a final lemma summarizing how $T'$ interacts with the tubes of $\frak T_1$.

\begin{lemma} \label{goodinters} With probability $(1 - K^-)$, there are $(1 - K^-) |\frak T|$ tubes $T' \in \frak T$ where the following holds:
 
\begin{enumerate}

\item There are at least $ K^{(1/2)-} D$ cubes $Q_j' \in X_{spaced}(T')$ where $I_{tr}(T', Q_j')$ holds.

\item Any $D$ tube segments of $T'$ between the cubes where $I_{tr}(T', Q_j')$ holds will intersect $\le K^{-(1/2)+} N$ cubes of $X$.

\item There are at most $K^+ D$ cubes $Q_j' \in X_{spaced}(T')$ where $I_{bad}(T', Q_j')$ holds. 

\end{enumerate}
\end{lemma}

Step 4. Let $T'$ be a tube obeying (1) - (3) from Lemma \ref{goodinters}.  If $I_{tr}(T', Q_j')$ holds and $I_{bad}(T', Q_j')$ does not hold, then we know that $P$ cuts $Q_j'$ at scale $K^-$.  Call such cubes good cubes.  Applying the vanishing lemma for tubes, Lemma \ref{vanishtube}, we see that $P$ cuts at scale $K^-$ on every cube intersecting $T'$ except for $K^{+} D$ bad tube segments between the good cubes $Q_j'$.  By (3) above, these bad tube segments can be covered by $\le K^+ D$ tube segments between the cubes $Q_j'$ where $I_{tr}(T', Q_j')$ holds.  By (2), these $K^{+} D$ tube segments intersect at most $K^{-(1/2) +} N$ cubes of $X$.  Therefore, $P$ cuts a fraction $(1 - K^{-(1/2)+})$ of all the cubes of $X$ that intersect $T'$.  This analysis holds for $(1 - K^-) |\frak T|$ tubes $T' \in \frak T$.

Finally, since each cube of $X$ intersects essentially the same number of tubes of $\frak T$, we see that $P$ cuts $(1 - K^-) |X|$ cubes of $X$ at scale $K^-$.
 \end{proof}

\section{Background in integral geometry} \label{revintgeom}

The Crofton formula plays an important role in studying the geometry of algebraic varieties.  It connects the $k$-dimensional volume of a surface $\Sigma^k \subset \RR^n$ with the number of intersection points between $\Sigma$ and various $(n-k)$-planes in $\RR^n$.  Let $AG(n-k, n)$ denote the affine Grassmannian of all affine $(n-k)$-planes in $\RR^n$.  The group of rigid motions of $\RR^n$ acts transitively on $AG(n-k, n)$.  Up to scaling, there is a unique invariant measure $\mu$ on $AG(n-k,n)$.  See \cite{S} for more details.  Let $| \pi \cap \Sigma|$ denote the cardinality of $\pi \cap \Sigma$.  

\begin{theorem} \label{crofton} (Cf. \cite{S}) For any $k,n$, there is a constant $C(k,n)$ so that for any k-dimensional submanifold $\Sigma^k \subset \RR^n$, 

$$ Vol_k (\Sigma) = C(k,n) \int_{AG(n-k,n)} |\pi \cap \Sigma| d\mu (\pi) . $$

\end{theorem}

The idea of the proof of the Theorem is as follows.  Define $Cr_k(\Sigma)$ to be the right hand side of the equation.  By choosing $C(k,n)$, we can arrange that the equation holds for the unit $k$-cube $[0,1]^k \times \{ 0 \}^{n-k} \subset \RR^n$.  Now $\Vol_k (\Sigma)$ and $Cr_k(\Sigma)$ are both invariant with respect to rigid motions, so the equality holds for any unit $k$-cube in 
$\RR^n$.  Both $\Vol_k$ and $Cr_k$ are linear with respect to disjoint unions, so the equation holds for any finite union of unit $k$-cubes.
A unit cube can be cut into $N^k$ cubes of side length $1/N$ for any integer $N$.  By symmetry, each of these cubes has $Cr_k$ equal to $N^{-k}$.   Therefore, the result holds for any $k$-cube of side-length $1/N$.  Assembling such cubes, it holds for any $k$-cube of rational side-length.  Also, $\Vol_k (\Sigma)$ and $Cr_k(\Sigma)$ are both monotonic, in the sense that if $\Sigma \subset \Sigma'$, then $Cr_k(\Sigma) \le Cr_k(\Sigma')$.  Therefore, the equation holds for any $k$cube.  Since $\Vol_k$ and $Cr_k$ are linear with respect to disjoint unions, the equation holds for any finite union of $k$-cubes.  This is already pretty good evidence for the theorem.  

For a smooth surface $\Sigma$, one can proceed roughly as follows.  One decompose an arbitrary smooth surface $\Sigma$ into small pieces that are almost $k$-cubes.  Such a small piece might be given by the graph of a function $h: [0,\delta]^k \rightarrow \RR^{n-k}$ with $|\nabla h| < \epsilon$.  In this situation, it suffices to prove that $|Cr_k(graph h) - \delta^k| \lesssim \epsilon \delta^k$.  We will give an analogous argument in the proof of Lemma \ref{intgeomavgest} below.

The Crofton formula leads to estimates on the volumes of algebraic varieties.

\begin{theorem} \label{algvol} Suppose that $Z$ is a degree $D$ algebraic variety of dimension $k$ in $\RR^n$.  Let $Q$ be an $n$-dimensional cube of side length $S$.  Then 

$$ \Vol_k ( Z \cap Q) \lesssim_{k,n} D S^k.$$
\end{theorem}

Here is a sketch of the proof.  We decompose $Z$ as $Z_{smooth} \cup Z_{sing}$.  We first bound the volume of $Z_{smooth}$.  Below, we will show that the $k$-volume of $Z_{sing}$ is zero.  For $\mu$-almost every $(n-k)$-plane $\pi$, $\pi$ intersects $Z_{smooth}$ transversally.   (This can be proven using Sard's theorem.)  Since $Z$ has degree $D$, if $\pi$ intersects $Z$ transversally, $|Z \cap \pi| \le D$.  Let $Q$ be a cube of side length $S$.  We compare $Z \cap Q$ with the $k$-skeleton of $Q$ (the union of the $k$-faces of $Q$), denoted $Sk_k Q$.  We notice that if any $(n-k)$-plane $\pi$ intersects $Q$, then it must intersect one of the $k$-faces of $Q$.  Therefore, we get the following inequality: $Cr_k(Z \cap Q) \le D Cr_k (Sk_k Q)$.  By Crofton's formula, $Cr_k$ is equal to the $k$-volume, and we see $\Vol_k (Z \cap Q) \le D \Vol_k (Sk_k Q) \lesssim_{k,n} D S^k$.

On the other hand a $(k-1)$-dimensional algebraic variety such as $Z_{sing}$ must have $k$-dimensional volume zero.  Its smooth part is a $(k-1)$-dimensional manifold, which has $k$-volume zero, and its singular part is a $(k-2)$-dimensional algebraic variety, and we can proceed inductively. This finishes the sketch of the proof of Theorem \ref{algvol}. 

In Section \ref{sectgraininess}, we will use Theorem \ref{algvol} repeatedly in the proof of the graininess theorem.  We will also need another integral geometry estimate in a similar spirit, which we describe and prove here.  This estimate concerns the intersection of a surface and a random plane.

\begin{lemma} \label{intgeomavgest} Let $R \ge 1$.  Let $T_R \subset \RR^3$ be the cylinder $x_1^2 + x_2^2 < R^2$.  Let $\pi(a,b)$ be the plane defined by $x_1 + a x_2 = b$.  Let $a$ be chosen uniformly at random in $(-1/10, 1/10)$.  Let $b$ be chosen uniformly at random in $(-2 R, 2 R)$.  

Suppose that $\Sigma$ is a 2-dimensional submanifold contained in $T_R$, and $f$ is a non-negative smooth function on $\Sigma$.  Then, up to a factor $C(R)$, the following quantities agree:

$$ \int_{\Sigma} f  darea  \sim \Avg_{a,b} \int_{\Sigma \cap \pi(a,b)} f dlength. $$

\end{lemma}

In particular, if we take $f=1$, then we see that $\Area (\Sigma) \sim \Avg_{a,b} \Length (\Sigma \cap \pi(a,b) )$. 

\begin{proof} Let's define $Cr(\Sigma, f) := \Avg_{(a,b)} \int_{\pi(a,b) \cap_\epsilon \Sigma} f dlength$.  We are trying to prove that $Cr(\Sigma, f) \sim \int_\Sigma f$.  

A key special case is when $\Sigma = Q_r$ is a square of side length $r$, and $f$ is equal to 1.  In this case, $\int_\Sigma f = r^2$.  Now we evaluate $Cr(\Sigma, f)$ in this special case.  First we fix $a$, and suppose that the angle between the plane $\pi(a,b)$ and the square $Q_r$ is $\theta(a)$.  Now the measure of the set of $b \in (-2R, 2R)$ so that $\pi(a,b)$ meets $Q_r$ is $\sim r \sin \theta(a)\sim r \theta(a)$.  If $\pi(a,b)$ does meet $Q_r$, the length of the intersection is always $\lesssim r$ and usually $\sim r$.  Therefore,

$$ \Avg_b \int_{\Sigma \cap \pi(a,b)} f dlength \sim_R \theta(a) r^2$$.  

Therefore, $Cr(\Sigma, f) \sim_R r^2 \Avg_a \theta(a)$.  But $\Avg_{a \in (-1/10, 1/10)} \theta(a)$ is $\sim 1$.  This proves the result when $\Sigma$ is a square and $f$ is 1.  

Both $\int_\Sigma f$ and $Cr(\Sigma, f)$ are linear in $f$ and additive with respect to disjoint unions, so the result holds when $\Sigma$ is any union of flat squares and $f$ is constant on each square.

Morally, this shows that the lemma should be true.  We now include a fairly detailed proof explaining what to do when $f$ is non-constant and what to do when $\Sigma$ is curved.

Next consider an arbitrary continuous $f$ on a square.  We can write $f_1 \le f \le f_2$ where $f_1$ and $f_2$ are sums of characteristic functions of squares with $\int_\Sigma f_1 \sim \int_\Sigma f_2$.  Then we see that $\int_\Sigma f_1 \sim Cr(\Sigma, f_1) \le Cr(\Sigma, f) \le Cr(\Sigma, f_2) \sim \int_\Sigma f_2$, and so $\int_\Sigma f \sim Cr(\Sigma, f)$.

We now see that $\int_\Sigma f \sim Cr(\Sigma, f)$ when $f$ is any union of squares and $f$ is continuous.

Next we will prove the theorem for a compact smooth surface $\Sigma$ and a smooth function $f$ supported on the interior of $\Sigma$.   Here we start to deal with the curvature of $\Sigma$.  The trickiest part to control is where $\pi(a,b)$ is nearly tangent to $\Sigma$.  To set aside this more delicate situation, we make the following definitions.

We define $\pi(a,b) \cap_\epsilon \Sigma \subset \pi(a,b) \cap \Sigma$ as the set of points $x \in \pi(a,b) \cap \Sigma$ where $\pi(a,b)$ and $T_x \Sigma$ make an angle $> \epsilon$.   Then we define 

$$Cr_\epsilon( \Sigma, f) = \Avg_{(a,b)} \int_{\pi(a,b) \cap_\epsilon \Sigma} f dlength. $$

For all $\epsilon < (1/1000)$, the same argument as above shows that $\int_\Sigma f \sim Cr_\epsilon (\Sigma, f)$ when $\Sigma$ is a union of squares and $f$ is continuous.

We will show that $Cr_\epsilon (\Sigma, f) \sim \int_\Sigma f$ for each sufficiently small $\epsilon$.  For almost every $(a,b)$, $\pi(a,b)$ intersects $\Sigma$ transversally, and so $Cr_0(\Sigma,f) = \lim_{\epsilon \rightarrow 0} Cr_\epsilon (\Sigma,f)$.  So it suffices to show $Cr_\epsilon(\Sigma, f) \sim \int_\Sigma f$ for all sufficiently small $\epsilon$.

Let $\psi_j$ be a partition of unity on $\Sigma$, where each $\psi_j$ is supported in a ball of radius $< \delta/5$, where $\delta$ is a small number depending on $\epsilon$ that we will choose below.  The support of $\psi_j$ is contained in a graph over a square, say $h: Q_\delta \rightarrow \RR$, where $Q_\delta$ is a square of side-length $\delta$ that intersects $\Sigma$ tangentially.  Since $\Sigma$ is compact, the second fundamental form of $\Sigma$ is uniformly bounded.  We write $X \lesssim 1$ if $X$ is bounded by a constant independent of $\epsilon, \delta$.   The second fundamental form is $\lesssim 1$, and so $|\nabla h| \lesssim \delta$ and $|h| \lesssim \delta^2$.  Also, we can assume that $|\nabla f| \lesssim 1$.  We let $f_j = \psi_j f$, and we can assume that $| \nabla f_j | \lesssim \delta^{-1}$.

Define a function $\bar f_j: Q_\delta \rightarrow \RR^{\ge 0}$ so that $\bar f_j (x) = f_j (h (x))$ for all $x \in Q_\delta$.  Because $\nabla h$ is small, $\int_{Q_\delta} \bar f_j \sim \int_{\Sigma} f_j$.  We already know that $\int_{Q_\delta} \bar f_j \sim Cr(Q_\delta, \bar f_j) \sim Cr_\epsilon(Q_\delta, \bar f_j)$ for all $\epsilon < 1/1000$.

Next we will prove that if $\delta$ is much smaller than $\epsilon$, then $Cr_\epsilon(Q_\delta, \bar f_j)$ approximately agrees with $Cr_\epsilon(\Sigma, f_j)$.  More precisely, if $\delta$ is much smaller than $\epsilon$, then we will prove:

$$ Cr_{2 \epsilon}(\Sigma, f_j)   \lesssim Cr_{\epsilon}(Q_\delta, \bar f_j) + \delta^{2.01}. \eqno{(1)}$$

$$ Cr_{2 \epsilon}(Q_\delta, \bar f_j) \lesssim Cr_{\epsilon}(\Sigma, f_j) + \delta^{2.01}. \eqno{(2)}$$

We already know that $Cr_\epsilon(Q_\delta, \bar f_j) \sim \int_{Q_\delta} \bar f_j \sim \int_{\Sigma} f_j$.  Plugging this into $(1)$ and $(2)$ and summing over the partition of unity, we see that for all $\epsilon < (1/2000)$, $Cr_\epsilon(\Sigma, f) \lesssim \int_\Sigma f + \delta^{.01}$ and $\int_\Sigma f \lesssim Cr_\epsilon(\Sigma, f) + \delta^{.01}$.  Taking $\delta \rightarrow 0$, we get $Cr_\epsilon(\Sigma, f) \sim \int_\Sigma f$.  

Now we prove inequality $(1)$.  The proof of $(2)$ is similar.  
Let $\Sigma_j \subset \Sigma$ be the support of $\psi_j$.  The probability that a plane $\pi(a,b)$ intersects either $Q_\delta$ or $\Sigma_j$ is $\lesssim \delta$.  It now suffices to prove the following estimate for each $\pi(a,b)$:

$$ \int_{\pi(a,b) \cap_{2 \epsilon} \Sigma_j} f_j dlength \lesssim \int_{\pi(a,b) \cap_\epsilon Q_\delta} \bar f_j + \delta^{1.01}.  \eqno{(1')} $$

If the left-hand side is zero, we are done, so we can suppose that $\pi(a,b)$ intersects $\Sigma_j$ some point at angle $> 2 \epsilon$.  Since $|\nabla h| < \delta$ is much smaller than $\epsilon$, $\pi(a,b)$ also intersects $Q_\delta$ at angle $> \epsilon$.  The intersection $\pi(a,b) \cap Q_\delta$ is a line segment $l_\delta$, and $\pi(a,b) \cap \Sigma_j$ is contained in the graph of a function $g: l_\delta \rightarrow \RR^2$.  By the geometry of the situation, we have $|\nabla g| \lesssim \epsilon^{-1} \delta$ and $|g| \lesssim \epsilon^{-1} \delta^2$.

Because of our bound on $|\nabla g|$, $ \int_{\pi(a,b) \cap \Sigma_j}  f_j d length \lesssim \int_{l_\delta} f_j (g(x)) dx $.
At each point $x \in l_\delta \subset Q_\delta$, we have $|f_j (g(x)) - \bar f_j(x)| = |f_j (g(x)) - f_j (h(x))| \lesssim \delta^{-1} |g(x) - h(x)| \lesssim \epsilon^{-1} \delta$.  Putting it together we get

$$ \int_{\pi(a,b) \cap \Sigma_j} f_j dlength \lesssim \int_{l_\delta} \bar f_j + \epsilon^{-1} \delta^2. $$

Finally, we choose $\delta < \epsilon^2$, so the last term is $\lesssim \delta^{1.5}$, and this proves $(1')$ and hence $(1)$.  The proof of $(2)$ is similar.  This establishes our result when $\Sigma$ is a compact smooth surface with boundary and $f$ is a smooth function supported on the interior of $\Sigma$.  

The rest of the proof is a routine approximation argument.  Let $\Sigma$ be a possible non-compact surface and $f$ a smooth function on $\Sigma$.  Let $\phi_j$ be a sequence of smooth compactly supported cutoff functions on $\Sigma$, with $0 \le \phi_j \le 1$, with $\phi_j(x)$ increasing in $j$, and $\phi_j \rightarrow 1$ pointwise.  By the case we proved, $\int_{\Sigma} \phi_j f \sim Cr(\Sigma, \phi_j f)$ for each $j$ (with a uniform constant in the $\sim$).  By the monotone convergence theorem $\int_{\Sigma} \phi_j f \rightarrow \int_{\Sigma} f$, and $Cr(\Sigma, \phi_j f) \rightarrow Cr(\Sigma, f)$.  

\end{proof}

\section{Planiness and graininess estimates} \label{sectgraininess}

In this section, we use degree reduction as a tool to prove our planiness estimate Proposition \ref{planypropintro}, and our graininess estimate Theorem \ref{graininessintro}.   Let us recall these results.  They hold for sets of tubes and cubes obeying certain hypotheses.

\begin{hypoth}  \label{uniform3trans} Let $E > 1$.  
Suppose that $\frak T$ is a set of tubes in $\RR^3$ with radius 1 and length $EN$, contained in a ball of radius $EN$.  Suppose that $X$ is a set of $N^{3 - \sigma}$ disjoint unit cubes in this ball.  Suppose that $X$ and $\frak T$ obey the following conditions:

\begin{enumerate}

\item Each tube $T \in \frak T$ intersects between $N$ and $E N$ cubes of $X$.

\item Each cube of $X$ intersects between $\rho$ and $E \rho$ tubes of $\frak T$, for some $\rho \ge 3$.

\item Each point of $\RR^3$ lies in $\le E \rho$ tubes of $\frak T$.

\item (At least three directions of tubes at each point) For each cube $Q \in X$, and for any two unit vectors $v_1, v_2 \in \RR^3$, at least a fraction $E^{-1}$ of the tubes of $\frak T$ that intersect $Q$ have angle $\ge E^{-1}$ with both $v_1$ and $v_2$.

\end{enumerate}

\end{hypoth}

Our planiness estimate is the following:

\begin{prop} \label{planyprop} Assume Hypotheses \ref{uniform3trans}.  Let $\epsilon > 0$.  For each cube $Q \in X$, we can choose a plane $\pi(Q)$ through $Q$, so that for a fraction $(1 - \epsilon)$ of cubes $Q \in X$, for a fraction $(1 - \epsilon)$ of the tubes $T \in \frak T$ that meet $X$, $\Angle(v(T), \pi(Q)) \le \Poly(E, \epsilon^{-1}) N^{-\sigma}$.
\end{prop}

Our graininess estimate controls how the plane $\pi(Q')$ rotates as we vary $Q'$ within a segment of a tube $T$.  

\begin{theorem} \label{graininess} Assume Hypotheses \ref{uniform3trans}.  Let $\epsilon > 0$.  
Also assume that $N^\sigma$ is larger than some large polynomial in $E, \epsilon^{-1}$.
Then there is a large constant $K = \Poly(E, \epsilon^{-1})$ so that the following holds.
For a fraction $(1 - \epsilon)$ of intersecting pairs $(Q,T) \in X \times \frak T$, for a fraction $(1 - \epsilon)$ of the cubes $Q'$ of $X$ which intersect $T$ with $\Dist(Q,Q') \le K^{-1} N^\sigma$,

$$ \Angle (\pi(Q), \pi(Q')) \le K N^{- \sigma}. $$

\end{theorem}

Both the results are proven by modelling $X$ by a polynomial surface of controlled degree $Z(P)$ using the degree reduction result Theorem \ref{degredtube}.  

We let $K = C (E \epsilon^{-1})^A$ for some large numbers $C, A$ that we can choose as needed.  We let $K^+$ denote a small positive power of $K$ that can change from line to line, and we let $K^-$ denote a small negative power of $K$ that can change from line to line.  In every occurence, $K^+ \ge (E \epsilon^{-1})^{10}$, and similarly $K^- \le (E \epsilon^{-1})^{-10}$.  On the other hand, in each occurence, $K^+ \le K^{1/100}$, so that $K^{(1/4)-}$ is always bigger than 1.

By Theorem \ref{degredtube}, we can find a polynomial $P$ of degree $\le K^+ N^{1 - \sigma}$ that cuts $(1 - K^-) |X|$ cubes of $X$ at scale $K^-$.  Using this degree bound, we will study the geometry of $Z(P)$ and use it to prove our results about the geometric structure of $X$ and $\frak T$. 

The proof of Proposition \ref{planyprop} is based on studying the tangent planes of $Z(P)$.  It will turn out that  $\pi(Q)$ is well approximated by the tangent plane $T_x Z(P)$ for most $x \in Z(P) \cap Q$.

The proof of Theorem \ref{graininess} is based on controlling the curvature of $Z(P)$.  Essentially we will show that the curvature has size $\le K^+ N^{- 2 \sigma}$ at many points.  

In order to carry out this plan, we will have to prove a sequence of estimates on the geometry of $Z(P)$.  It slightly simplifies matters to know that $Z(P)$ is smooth and irreducible.  We can assume this without loss of generality for the following reason.  Using Theorem \ref{degredtube}, we saw that there is a polynomial $P_0$ of degree $\le K^+ N^{1- \sigma}$ that cuts $(1 - K^-)$ of the cubes of $X$ at scale $K^-$.  If a polynomial $P$ lies in a tiny neighborhood of $P_0$, then $P$ cuts all the same cubes at a slightly larger scale.  This happens because we can arrange that the set of points $x \in B^{K^+ N}$ where the sign of $P$ differs from the sign of $P_0$ has volume less than $K^{-10}$ by taking the neighborhood small enough.  Therefore, we can choose a generic polynomial $P$ in some tiny ball in the space of polynomials.  In this way we can arrange that $\nabla P$ is non-vanishing on $Z(P)$ and so $Z(P)$ is a smooth surface.  By the same genericity argument, we can assume that $P$ is irreducible.

We let $Z$ denote $Z(P) \cap B(K^+ N)$, the part of $Z(P)$ in the ball containing $X$.  

Here is an outline of this section.  In Section 4.1, we study the geometry of $Z(P)$ in a typical cube $Q \in X$.  We prove that $Z(P) \cap Q$ resembles a union of nearly-parallel planes.  The plane $\pi(Q)$ is an approximation of the tangent plane of these planes.  In this section, we prove Proposition \ref{planyprop}.  In Section 4.2, we study the geometry of $Z(P)$ on a typical segment $Seg$ of a tube $T \in \frak T$ of length $\sim N^\sigma$.  In particular, we start to focus on the second fundamental form of $Z(P)$, and we prove that the second fundamental form is morally constant away from a small set of bad curves in $Z(P)$.  In Section 4.3, we consider the intersection of $Z(P) \cap Seg$ with a random plane -- the resulting slice $\Gamma$ avoids the bad curves, and so the second fundamental form of $Z(P)$ is morally constant along such a slice.  In Section 4.4, we begin to prove curvature estimates.   A unit vector $v \in T_x Z$ is called straight if the second fundamental form of $Z$ vanishes in the direction $v$.  If the direction of the tube $T$ is far from straight at some points of the slice $\Gamma$, then we get good curvature estimates for $Z$ on the set $\Gamma$.  
In Section 4.5, we prove that for most $Q \in X$, at most points of $Z(P) \cap Q$, the second fundamental form is bounded by roughly $N^{-2 \sigma}$.  The key observation here is that each point $x \in Z(P)$ lies in three different tubes of $\frak T$ in quantitatively different directions.  At most two of these directions can be straight and at least one must be far from straight.  Using the tube in the far from straight direction, and applying the bounds from Section 4.4, we get a curvature estimate at almost all $x \in Z(P) \cap Q$.  
In Section 4.6, we use this curvature bound to control how the tangent plane $\pi(Q)$ rotates as $Q$ slides along a segment of $T$, and we prove Theorem \ref{graininess}.

\subsection{Reasonable cubes} \label{secreascube}

We say that a condition on a cube $Q$ is a reasonable cube condition if it holds for $(1 - K^-) |X|$ cubes $Q \in X$.  
When we defined $P$ above (using Theorem \ref{degredtube}), we saw that $P$ cuts $Q$ at scale $K^-$ for $(1 - K^-) |X|$ cubes $Q \in X$.  Thus we get:

\begin{reascube} \label{cuts} The polynomial $P$ cuts $Q$ at scale $K^-$.
\end{reascube}

The next condition involves the normal vector.  Let $N := \nabla P / |\nabla P|$ be the unit normal vector to $Z(P)$.  The vector $N$ is defined everywhere on $Z(P)$, because $\nabla P$ is non-vanishing on $Z(P)$.  For a tube $T \in \frak T$, let $v(T)$ be a unit vector parallel to the axis of $T$.  The vector $v(T)$ is well-defined up to sign, and we make an arbitrary choice for each tube $T$.

\begin{lemma} \label{cylinderestimate} If $T_R$ is any cylinder of radius $R$ and infinite length, then the following estimate holds.  

$$ \int_{x \in Z(P) \cap T_R} |v(T_R) \cdot N(x)| dx \le \pi R^2 \Deg P. $$

\end{lemma}

This estimate is Lemma 2.1 in \cite{Gu1}.  The idea is that the integral on the left hand-side is the area (counted with multiplicity) of the projection of $Z \cap T_R$ onto a cross-section of $T_R$.  This projection covers almost every point of the cross-section at most $\Deg P$ times, because a line intersects $Z(P)$ at most $\Deg P$ times unless the line lies in $Z(P)$.  Also the cross-section is a disk of radius $R$. So the area of the projection counted with multiplicity is at most $\pi R^2 \Deg P$.

We will sometimes want to discuss fatter versions of tubes $T \in \frak T$ or cubes $Q \in X$.  For a tube $T \in \frak T$, we let $T^+$ be the concentric cylinder of radius 100 instead of radius 1.  For a cube $Q \in X$, we let $Q^+$ be the concentric cube of side length 1000 instead of side length 1.

\begin{reascube} \label{plany} 

$$ \Avg_{T \in \frak T, T \textrm{ meets } Q^+} \left( \int_{Z \cap Q^+} |v(T) \cdot N(x)| \right) \le K^+ N^{- \sigma}. $$

\end{reascube}

\begin{proof}  Fix any cylinder $T \in \frak T$.  We apply Lemma \ref{cylinderestimate} to the concentric cylinder around $T$ with radius 200.  Any cube $Q$ so that $Q^+$ meets $T$ lies in this larger cylinder.  Also, the cubes of $X$ are disjoint, and so each point lies in $O(1)$ of the cubes $Q^+$.  So we get the following estimate.

$$ \sum_{Q \in X, Q^+ \textrm{ meets } T} \int_{x \in Z \cap Q^+} |v(T) \cdot N(x)| dx \le K^+ N^{1 - \sigma}. $$

By hypothesis, there are $\ge N$ cubes $Q \in X$ that meet $T$.  Therefore, for each $T$, we get

$$ \Avg_{Q \in X, Q^+ \textrm{ meets } T}  \int_{x \in Z \cap Q^+} |v(T) \cdot N(x)| dx \le K^+ N^{- \sigma}. $$

Since this holds for every $T \in \frak T$, it also holds when we average over $T \in \frak T$.  We get

$$\Avg_{T \in \frak T} (\Avg_{Q \in X, Q^+ \textrm{ meets } T}  \int_{x \in Z \cap Q^+} |v(T) \cdot N(x)| dx) \le K^+ N^{- \sigma}. $$

Since each tube has essentially the same number of cubes, and each cube lies in essentially the same number of tubes, changing the order of the two averages can only increase the right-hand side by a factor $K^+$.  Therefore, for $(1- K^-) |X|$ cubes $Q \in X$, we have

$$ \Avg_{T \in \frak T, T \textrm{ meets } Q^+}  \int_{x \in Z \cap Q^+} |v(T) \cdot N(x)| dx \le K^+ N^{- \sigma}. $$

\end{proof}

Next we prove that the normal vector is nearly constant (in an average sense) on $Z \cap Q$.

\begin{reascube} \label{tangplane} There is a plane $T_Q Z$ so that
$\int_{Z \cap Q^+} \Angle(T_x Z, T_Q Z) \le K^+ N^{-\sigma}$ and  $ \Avg_{x \in Z \cap Q^+} \Angle(T_x Z, T_Q Z) \le K^+ N^{-\sigma}. $

\end{reascube}

\begin{proof} By the transversality hypothesis, we can choose tubes $T_1$, $T_2$ in $\frak T$ which meet $Q$ so that the angle between $v(T_1)$ and $v(T_2)$ is $\ge E^{-1}$ and so that for both tubes $T_i$, 
$  \int_{Z \cap Q^+} |v(T_i) \cdot N(x)| \le K^+ N^{- \sigma}. $  If $T_Q Z$ is the plane spanned by $v(T_1)$ and $v(T_2)$, then we get $\int_{Z \cap Q^+} \Angle(T_x Z, T_Q Z) \le K^+ N^{-\sigma}$.  
On the other hand, by Reasonable Cube Condition \ref{cuts}, $P$ cuts $Q$ at a small scale, and so $\Area Z \cap Q \ge 1$, so we can bound the average by the integral.
\end{proof}

For each reasonable cube $Q$, we pick a plane $T_Q Z$ obeying the condition of the lemma.  The plane $T_Q Z$ is well-defined up to a rotation by angle $\le K^+ N^{-\sigma}$ - within this small range of possibilities we make an arbitrary choice.

\begin{reascube} \label{tubeplaneangle} $ \Avg_{T \in \frak T, T \textrm{ meets } Q^+} \Angle(v(T), T_Q Z) \le K^+ N^{- \sigma} . $

\end{reascube}

\begin{proof} For any $x \in Z \cap Q+$, $\Angle(v(T), T_Q Z) \le \Angle (v(T), T_x Z) + \Angle(T_x Z, T_Q Z)$.  We want to study the average size of $\Angle(v(T), T_Q Z)$ over all $T \in \frak T$ that meet $Q^+$.  Reasonable Cube Condition \ref{plany} says that the average size of the first angle is $\le K^+ N^{- \sigma}$.  Reasonable Cube Condition \ref{tangplane} says that the average size of the second angle is $\le K^+ N^{-\sigma}$.  Combining the bounds, we get:

$$ \Avg_{T \in \frak T, T \textrm{ meets } Q^+} \Angle(v(T), T_Q Z) \le K^+ N^{- \sigma} . $$

\end{proof}

This result immediately implies our planiness estimate, Proposition \ref{planyprop}: we take $\pi(Q)$ to be $T_Q Z$.

Since $Z$ cuts any reasonable cube $Q$, we know that $\Area (Z \cap Q) \ge 1$ for any reasonable cube.  We can also show that, for a reasonable cube, the area is not larger.

\begin{reascube} $\Area(Z \cap Q^+) \le K^+$
\end{reascube}

\begin{proof}  By the Crofton formula (see Theorem \ref{algvol}), the area of $Z$ in our ball of radius $\le K^+ N$ is at most $K^+ N (\Deg P)^2 \le K^+ N^{3 - \sigma}$.  The number of cubes $Q \in X$ is $N^{3 - \sigma}$.  The cubes $Q$ are disjoint, and the cubes $Q^+$ overlap with bounded multiplicity.  Therefore, there are at most $K^- |X|$ cubes $Q$ so that $\Area (Z \cap Q^+) \ge K^+$.  
\end{proof}

Next we study more closely the geometry of $Z \cap Q$.  For a reasonable cube $Q$, we will prove that $Z \cap Q$ consists of a union of nearly flat disks with small holes cut out of them and with a surface of small area glued in.  As far as I know, this piece of small area may include thin tubes connecting one of the disks to another as well as stalagmites and stalagtites sticking up and down from the disks, and it may have non-trivial topology.
Let us formulate this result precisely.

The geometry of $Z$ is nicest in a cylinder around $Q$ described as follows.
 We choose (orthogonal) coordinates $(x_1, x_2, x_3)$ so that the origin is the center of $Q$ and $T_Q Z$ is the $(x_1,x_2)$-plane.  Then we let $Cyl_H(Q)$ be cylinder defined by equations $x_1^2 + x_2^2 < 100$, and $|x_3| < H$.  We focus on $H$ in the range $[10,20]$, so that we always have $Q \subset Cyl_H(Q) \subset Q^+$.  Now for most $H \in [10,20]$, we will prove that $Z \cap Cyl_H(Q)$ has the following structure.

\begin{reascube} \label{paramdisk} Let $\lambda = K^{-10}$.   For most $H \in  [10,20]$, the following holds.
 
\begin{enumerate}

\item There exist functions $f_j: B^2(10) \rightarrow (-H,H)$ with Lipschitz constant $\le 10 \lambda$.

\item There is a finite set of disjoint ``bad'' balls $B_i \subset B^2(10)$ with the sum of the radii at most $K^+ \lambda^{-2} N^{-\sigma}$.  We define $Y := B^2(10) \setminus (\cup_i \bar B_i)$.

\item The graph of each function $f_j: Y \rightarrow (-H,H)$ lies in $Z \cap Q^+$.  

\item The graphs are close together in the sense that for each $y \in B^2(10)$ and each $h \in [-10,10]$, there exists a $j$ so that $|f_j(y) - h| \le K^-$.  (Therefore, the number of graphs $f_j$ is $\ge K^+$.)

\item The graphs are also disjoint and maintain their order in the following sense: at each $y \in Y$, $f_1(y) < f_2(y) < ...$, and at each $y \in B^2(10)$, $f_1(y) \le f_2(y) \le ...$.

\item The part of $Z \cap Cyl_H(Q)$ outside of the graphs of $f_j: Y \rightarrow \RR$ has area $\le K^+ \lambda^{-1} N^{- \sigma}$.

\end{enumerate}

\end{reascube}

Remark: This result actually holds for a range of $\lambda$, but taking $\lambda = K^{-10}$ is a good choice for our applications below.

We define $Z_{Q, nice} \subset Z \cap Cyl_H(Q) \subset Z \cap Q^+$ to be the union of the graphs of $f_j$ over $Y$.  For $x \in Z_{Q,nice}$, we have $\Angle( T_x Z, T_Q Z) \le 10 \lambda \le K^{-10+}$.  

\begin{proof}

In the proof of Reasonable Cube Condition \ref{paramdisk}, it helps to better understand how the plane $T_x Z$ varies for $x \in Q^+$.    Let $w$ be a unit vector in $\RR^3$.  Consider the set $Tan(w) := \{ x \in Z | \nabla P(x) \cdot w = 0 \}$.  This is the set of points $x \in Z$ where $w \in T_x Z$.

\begin{lemma} \label{tan(w)est} For each $w \in S^2$, the set $Tan(w) \subset Z$ is a curve of length $\le K^+ N^{3-2 \sigma}$.
\end{lemma}

\begin{proof} $Tan(w)$ lies in the variety defined by the two equations: $P(x) = 0$ and $w \cdot \nabla P(x) = 0$.  Since $P$ is irreducible, either this variety is all of $Z(P)$ or else it is an algebraic curve of degree $\le (\Deg P)^2$.  If this variety is all of $Z(P)$, then $Z(P)$ is a cylinder.  This doesn't occur for generic $P$, so we can ignore it.  

By the Crofton formula, an algebraic curve of degree $D$ in $B^3(K^+ N)$ has length $\le K^+ N D$.  In our case, the length is $\le K^+ N (\Deg P)^2 \le K^+ N^{3 - 2 \sigma}$.
\end{proof}

Let $W_\lambda$ denote a $\lambda$-net of points in $S^2$, with $|W| \sim \lambda^{-2}$.  We let $Tan(W_\lambda) := \cup_{w \in W_\lambda} Tan(w)$. The total length of $Tan(W_\lambda)$ is $\le \lambda^{-2} K^+ N^{3 - 2 \sigma}$.  Since there are $N^{3 - \sigma}$ cubes $Q \in X$, a reasonable $Q$ obeys the following estimate:

\begin{reascube} \label{Tan(W_lambda)bound} The length of $Q^+ \cap Tan(W_\lambda)$ is $\le \lambda^{-2} K^+ N^{-\sigma}$.
\end{reascube}

This condition says that $Tan(W_\lambda \cap Q)$ is almost empty.  To get a perspective, let's consider what would happen if it were empty.  If $Q^+ \cap Tan(W_\lambda)$ were empty, then $T_x Z$ would be nearly constant on each component of $Z \cap Q^+$.  If $Q^+ \cap Tan(W_\lambda)$ were empty, then the normal vector $N(x)$ would never be perpendicular to any $w \in W_\lambda$.  The unit vectors normal to a fixed $w \in W_\lambda$ form a great circle $w^\perp$ on $S^2$.  We let $W_\lambda^\perp := \cup_{w \in W_\lambda} w^\perp$.  The complement $S^2 \setminus W_\lambda^\perp$ is a union of open cells of diameter $< 2 \lambda$.  Therefore, if $Q \cap Tan(W_\lambda)$ were empty, then on each connected component of $Z \cap Q$, the normal vector $N(x)$ could vary by at most $2 \lambda$.  

I believe that $Tan(w) \cap Q$ is small but may be non-empty for all cubes $Q \in X$.
$Tan(w)$ is an algebraic curve of degree $\le (\Deg P)^2 \le K^+ N^{2 - 2 \sigma}$.  Such a curve may have as many as $N^{4 - 4 \sigma}$ connected components, and so a reasonable cube $Q$ may contain $\sim N^{1 - 3 \sigma}$ connected components of $Tan(w)$.  If $\sigma < 1/3$, especially if $\sigma$ is close to zero, I suspect that $Tan(w) \cap Q$ may contain a large number of very short curves.

Let $N(Q)$ be the unit vector normal to $T_Q Z$.  Let $G_0 = G_0 \subset S^2$ be a small neighborhood of $N(Q)$, whose boundary lies in $W_\lambda^\perp$.  We can arrange that $G_0$ contains the $(1/10) \lambda$ neighborhood of $N(Q)$, and is contained in the $5 \lambda $-neighborhood of $N(Q)$, and that $\partial G_0 \subset W_\lambda^\perp$.  (If $N(Q)$ is not too close to $W_\lambda^\perp$, then $G_0$ is a single component of $S^2 \setminus W_\lambda^\perp$.  But if $N(Q)$ is within $(1/10) \lambda$ of $W_\lambda^\perp$, then $G_0$ must contain two or more components.)
We let $G := \{ x \in Z \cap Q^+ | N(x) \in G_0 \}$.  The letter $G$ stands for `good'  - these are the points of $Z \cap Q^+$ with good tangent planes.  We define $B:= (Z \cap Q^+) \setminus G$.

\begin{lemma} \label{boundaryG} $\Length(\partial G \cap Q^+) \le K^+ \lambda^{-2} N^{-\sigma}$.
\end{lemma}

\begin{proof} We have $\partial G \subset Tan(W_\lambda)$. \end{proof}

\begin{lemma} $\Area B \le K^+ \lambda^{-1} N^{-\sigma}.$ 
\end{lemma}

\begin{proof} For $x \in B$, $\Angle (T_x Z, T_Q Z) \ge (1/10) \lambda$.  But by Reasonable Cube Condition \ref{tangplane}, we have 
$\int_{Z \cap Q^+} \Angle (T_x Z, T_Q Z) \le K^+ N^{-\sigma}$.
\end{proof}

At this point, we exploit the geometry of $Cyl_H(Q)$.  The boundary of $Cyl_H(Q)$ consists of a top and bottom (defined by $x_3 = \pm H$) and the side (defined by $x_1^2 + x_2^2 = 100$).  By choosing $H$ generically, we can arrange that the intersection of $Z$ with the top and bottom are small.

$$ \Avg_{H \in [10,20]} \Length( Z \cap \textrm{ top and bottom of } Cyl_H(Q)) \le \int_{Z \cap Q^+} |\Angle(T_x Z, T_Q Z)| \le K^+ N^{- \sigma}. $$

Therefore, for all $H \in [10,20]$ except for a subset of length $K^-$, the following Lemma holds.

\begin{lemma} \label{topbottomcyl}
$\Length( Z \cap \textrm{ top and bottom of } Cyl_H(Q)) \le K^+ N^{-\sigma}$.
\end{lemma}

From now on, we restrict to $H \in [10,20]$ where Lemma \ref{topbottomcyl} holds.  

We define the bad curves to be $\partial G \cap Cyl_H(Q)$ together with the intersection of $G$ with the top and bottom of $Cyl_H(Q)$.  By Lemma \ref{boundaryG} and Lemma \ref{topbottomcyl} their total length is $\le K^+ \lambda^{-2} N^{-\sigma}$.  We let $\pi: Cyl_H(Q) \rightarrow B^2(10)$ be the projection $(x_1, x_2, x_3) \rightarrow (x_1, x_2)$.  The projection of the bad curves still has total length $\le K^+ \lambda^{-2} N^{-\sigma}$.

\begin{lemma} \label{badballsradii} The projection of the bad curves can be covered by finitely many disjoint balls $B_i$ with the sum of the radii at most $K^+ \lambda^{-2} N^{-\sigma}$.
\end{lemma}

\begin{proof} The bad curves are a union of finitely many connected components $\gamma_i$.  The projection of $\gamma_i$ is contained in a ball of radius $r_i \le \Length(\gamma_i)$.  So we can cover all the projections by balls with the sum of the radii bounded by the total length of the bad curves, which is at most $K^+ \lambda^{-2} N^{-\sigma}$.

These balls may not be disjoint.  But if two balls of radii $r_1$ and $r_2$ intersect, they may be covered by one ball of radius $r_1 + r_2$.  So in our list of balls, we can replace two intersecting balls with one larger ball maintaining our bound on the sum of the radii.  Doing this repeatedly, we arrive at a collection of disjoint balls where the sum of the radii obeys the desired bound.  \end{proof}

Let $Y$ be $B^2(10) \setminus (\cup_i \bar B_i)$.  Since the balls $B_i$ are disjoint, $Y$ is connected.
Here we removed the closed balls $\bar B_i$ so that $Y$ is an open set.  

We let $\pi: Cyl_H(Q) \rightarrow B^2(10)$ be the projection to the $(x_1, x_2)$ coordinates.  We let $G' := \{ x \in G \cap Cyl_H(Q) | \pi(x) \subset Y \}$.  Note that $G'$ is an open subset of $G$, so it is also a manifold.  We now prove that $\pi: G' \rightarrow Y$ is a covering map.  The map $\pi: G' \rightarrow Y$ is a local diffeomorphism because the tangent plane of $x \in G$ is close to the $(x_1,x_2)$-plane.  It just remains to check that the map $\pi: G' \rightarrow Y$ is a proper map.  In other words, we have to check that if $K \subset Y$ is compact, then $\pi^{-1}(K) \subset G'$ is also compact.  The map $\pi$ extends to the closure $\bar G'$, and $\pi^{-1}(K)$ is automatically a compact subset of $\bar G'$, and the issue is to check whether $\pi^{-1}(K)$ contains any boundary points of $\bar G'$.  
To check this, we have to prove that $\pi$ maps the boundary of $G'$ to the complement of $Y$.  The boundary of $G'$ has several types of curves: curves in $\partial G \cap Cyl_H(Q)$ are mapped to $\cup B_i$; curves in $G$ intersected with the top and bottom of $Cyl_H(Q)$ are mapped to $\cup_i B_i$; and curves in $G$ intersected with the sides of $Cyl_H(Q)$ are mapped to $\partial B^2(10)$.  Therefore, $\pi: G' \rightarrow Y$ is a proper map, and so it is a covering map. 

Now we study the map $\pi: G' \rightarrow Y$ using the structure of covering maps.  Since $Y$ is connected, the number of points in each preimage $\pi^{-1}(y)$ is constant.  Let the cardinality of the fibers be $\kappa$.  If we take a based loop in $Y$, we can look at the holonomy of the covering over the based loop.  The holonomy is a permutation of the points in the fiber over the base point of the loop.  But the vertical order of the points is preserved by the holonomy, and therefore the holonomy is the identity.  Therefore, $G'$ is the union of the graphs of $\kappa$ functions $f_j: Y \rightarrow (-H, H)$.  We can label the graphs so that at each point $y \in Y$, $f_1(y) < f_2(y) < ... $

We remark that we have not yet checked that $\kappa > 0$.  This is a somewhat tricky point.  We will prove below that $\kappa > K^+$.  

Because the tangent plane of each point $x \in G$ has angle $\le 5 \lambda$ with the $(x_1, x_2)$-plane, each function $f_j$ obeys $|\nabla f_j| \le 5 \lambda$.  

\begin{lemma} If $p, p'$ are two points in $Y$, then $|f_j(p) - f_j(p')| \le 10 \lambda |p - p'|$, where $|p-p'|$ is the Euclidean distance between $p, p'$ in $B^2(10)$.  
\end{lemma}

\begin{proof} Consider the segment $\gamma$ from $p$ to $p'$ in $B^2(10)$.  This segment intersects the balls $B_i$ in some disjoint smaller segments $\gamma \cap B_i$.  Replacing each segment $\gamma \cap B_i$ with a piece of the arc of the boundary of $B_i$, we get a curve $\tilde \gamma$ in $\bar Y$ of length at most $(\pi/2) |p- p'|$.  Perturbing the curve a bit, we get a curve from $p$ to $p'$ in $Y$ of length $\le 2 | p - p'|$.  Now we integrate $\nabla f_j$ along this curve, and conclude $|f_j(p) - f_j(p')| \le 2 | p - p' | \cdot 5 \lambda$.  \end{proof}

Now it follows that $f_j$ extends to a Lipschitz function from $B^2(10)$ to $(-30,30)$ with the Lipschitz constant $10 \lambda$.  The extension procedure is to define, for any $p \in B^2(10)$

$$f_j(p) := \max_{y \in Y} f_j(y) - 10 \lambda | y - p|. $$

If $p \in Y$, then the maximum on the right-hand side is achieved by $y=p$, and so the new definition agrees with the original definition of $f_j$ on $Y$.  It's standard to check that $f_j$ still obeys $|f_j(p) - f_j(p')| \le 10 \lambda |p - p'|$ for all $p, p' \in B^2(10)$.  Also,
since the $f_j$ obey $f_1 < f_2 < ...$ on $Y$, it follows that $f_1 \le f_2 \le ...$ on $B^2(10)$.

Next, we prove that the complement $Z \cap Cyl_H(Q) \setminus G'$ has small area.  The complement $Z \cap Cyl_H(Q) \setminus G'$ lies in the union of $B$ and the set $B' := \{ x \in G \cap Cyl_H(Q) | \pi(x) \in \cup_i \bar B_i \}$.  We already know that $\Area (B) \le K^+ \lambda^{-1} N^{-\sigma}$.  We will prove an area estimate for this latter set $B'$.  

\begin{lemma} For a reasonable cube $Q$, $\Area B' \le K^+ \lambda^{-4} N^{-2 \sigma}$.
\end{lemma}

\begin{proof} First we will estimate the length of $\partial B'$.  If $x \in \partial B'$, then either $x \in \partial G$ or $\pi(x) \in \partial B_i$ for some bad ball $B_i$ or $x$ lies in the top or bottom of the cylinder $Cyl_H(Q)$.  We deal with the parts separately.  

By Lemma \ref{boundaryG}, the length of $\partial G \cap Cyl_H(Q)$ is bounded by $K^+ \lambda^{-2} N^{-\sigma}$.

By Lemma \ref{topbottomcyl}, $\Length( Z \cap \textrm{ top and bottom of } Cyl_H(Q)) \le K^+ N^{-\sigma}$.

The boundary points with $\pi(x) \in \partial B_i$ lie in the graphs of the functions $f_j$.  
The number of layers $\kappa$ is controlled by the area of $Z \cap Cyl_H(Q)$ which is $\le K^+$.  The total length of the boundaries of the $B_i$ is controlled by the sum of the radii which is $\le K^+ (\lambda)^{-2} N^{-\sigma}$.  So the length of this part of boundary of $B'$ is also bounded by $K^+ \lambda^{-2} N^{-\sigma}$.

In total, $\Length \partial B' \le K^+ \lambda^{-2} N^{-\sigma}$.

Since $B' \subset G$, the tangent plane at each point of $B'$ is almost tangent to the $(x_1, x_2)$-plane.  We can now choose an orientation on (each component of) $B'$ so that 
$\Area B' \le 2 \int_{B'} dx_1 \wedge dx_2$.  We now evaluate this integral using Stokes theorem.  Let the boundary of $B'$ be the union of connected curves $\partial B'_a$.  We have

$$ \Area B' \le 2 \sum_a \int_{\partial B'_a} x_1 dx_2. $$

We choose $c_a$ to be the $(x_1)$-coordinate of  a point in $\partial B'_a$.  Therefore, $|x_1 - c_a| \le \Length \partial B'_a$ for all $x \in \partial B'_a$.  Since the integral over a closed curve of $c dx_2$ vanishes, we can rewrite the boundary integral as

$$  \sum_a \int_{\partial B'_a} (x_1 - c_a) dx_2 \le \sum_a (\Length \partial B'_a)^2 \le (\Length \partial B')^2 \le K^+ \lambda^{-4} N^{-2 \sigma}. $$

\end{proof} 

Since $N^\sigma$ is much larger than $\lambda^{-1} = K^{10}$, we have $\Area B' \le K^+ \lambda^{-4} N^{-2 \sigma} \le K^+ \lambda^{-1} N^{-\sigma}$.

Therefore, we get

\begin{lemma} \label{ZsetminusG'}
$ \Area Cyl_H(Q) \cap (Z \setminus G') \le K^+ \lambda^{-1} N^{-\sigma}.$
\end{lemma}

We are now ready to prove that $\kappa > 0$ so that the set of functions $f_j$ we have been studying is not empty.  

\begin{lemma} We have $\kappa > 0$.  Moreover, at each point $y \in B^2(1)$, for each $h \in [-H, H]$, there is some $j$ so that $|f_j(y) - h| \le K^-$.
\end{lemma}

\begin{proof} First we prove that $\kappa \ge 1$.  If $\kappa = 0$, then $G'$ would be empty.  By Lemma \ref{ZsetminusG'}, we would have $\Area Cyl_H(Q) \cap Z \le K^+ \lambda^{-1} N^{- \sigma}$.  But $P$ cuts $Q$ at scale $K^-$, and so $\Area Z \cap Q \ge 1$.  This contradiction shows that $\kappa \ge 1$.

We can apply the same argument to any ball of radius $\ge K^-$ in $Cyl_H(Q)$.  Since $P$ cuts $Q$ at scale $K^-$, the intersection of $Z$ with any such ball has area $\ge K^-$.  By Lemma \ref{ZsetminusG'}, the area of $Cyl_H(Q) \cap (Z \setminus G')$ is much smaller than $K^-$.  Therefore, $G'$ enters every ball of radius $K^-$ in $Cyl_H(Q)$.

So for any $y \in B^2(10)$, and any height $h \in (-H, H)$, there exists a point $y' \in Y$ with $| y - y'| \le K^-$ and a $j$ so that $|f_j(y') - h | \le K^-$.  Since $f_j$ is $10 \lambda$ Lipschitz, we see that $|f_j(y) - h| \le K^-$ as well.  \end{proof}

This finishes the proof of Reasonable Cube Condition \ref{paramdisk}. \end{proof}

We say that $Q \in X$ is a reasonable cube if it obeys Reasonable Cube Conditions 1 - \ref{paramdisk}.

\subsection{The curvature of $Z$ on reasonable tube segments}

In this subsection, we consider the geometry of $Z$ in reasonable tube segments of tubes $T \in \frak T$ of length $K^{-1} N^\sigma$.  For reasonable segments, we will eventually prove that $T_Q Z$ varies slowly along the segment.  Along the way, we will estimate the behavior of the normal vector and the curvature.

Given an intersecting pair $Q \in X$ and $T \in \frak T$, we define the tube segment $Seg(Q,T)$ as the segment of $T^+$ centered at $Q$ of length $K^{-1} N^\sigma$.  (Recall that $T^+$ is the concentric cylinder around $T$ with radius 100 instead of radius 1.)  

We say that a condition on $Seg(Q,T)$ is reasonable if it holds for a fraction $(1 - K^-)$ of intersecting pairs $(Q, T)$.  Up to a factor of $E$, we know that any two cubes lie in the same number of tubes, and any two tubes contain the same number of cubes.  Therefore, a condition on $Seg(Q,T)$ is reasonable if either of the following holds:

\begin{itemize}

\item For $(1 - K^-) |\frak T|$ tubes $T \in \frak T$, for a fraction $(1 - K^-)$ of the cubes $Q \in X$ that intersect $T$, the condition on $Seg(Q,T)$ holds.

\item For $(1 - K^-) |X|$ cubes $Q \in X$, for a fraction $(1 - K^-)$ of the tubes $T \in \frak T$ that intersect $Q$, the condition on $Seg(Q,T)$ holds.  

\end{itemize}

\begin{reasseg} $\Angle(v(T), T_Q Z) \le K^+ N^{-\sigma}$.
\end{reasseg}

\begin{proof} For $(1 - K^-)|X|$ cubes $Q \in X$, Reasonable Cube Condition \ref{tubeplaneangle} tells us that 

$$\Avg_{T \textrm{ meets } Q} \Angle(v(T), T_Q Z) \le K^+ N^{-\sigma}.$$

So for a fraction $(1 - K^-)$ of all $T$ that intersect $Q$, we have the desired estimate.
\end{proof}

To prove that some estimates hold on almost all intersecting pairs, $(Q,T)$, we will have to do some averaging.  In our hypotheses, we assumed some uniformity conditions on $X$ and $\frak T$, and these make the averages easier to understand.  In particular, the uniformity implies the following simple lemmas.

\begin{lemma} \label{avgbound} The following estimate holds at each point $x$:

$$\mu(x) := \Avg_{Q \in X, T \in \frak T, Q \textrm{ meets } T} \chi_{Seg(Q,T)} (x) \le K^{-1+} N^{-3 + 2 \sigma}. $$

\end{lemma}

\begin{proof} The number of intersecting pairs is at least $|X| \rho \ge N^{3 - \sigma} \rho$.  Fix a point $x$.  The number of segments $Seg(Q,T)$ containing the point $x$ is bounded as follows.  The number of tubes $T^+$ containing $x$ is at most $K^+ \rho$.  For each $T$ containing $x$, the number of $Q$ that lie within $K^{-1} N^\sigma$ of $x$ and intersect $T$ is $\le 10 K^{-1} N^\sigma$.  Therefore, for each $x$, the number of segments $Seg(Q,T)$ that contain $x$ is $\le K^{-1+} N^\sigma \rho$.

The density $\mu(x)$ is bounded by the quotient $(K^{-1+} N^\sigma \rho) / (N^{3-\sigma } \rho) = K^{-1+} N^{-3 + 2 \sigma}$.  \end{proof}

As a simple consequence, we can control the area of $Z \cap Seg(Q,T)$ for a reasonable segment.

\begin{reasseg} $\Area Z \cap Seg(Q,T) \le K^{-1+} N^\sigma$. 
\end{reasseg}

\begin{proof} 
$\Avg_{Q \textrm{ meets } T} \Area (Seg(Q,T) \cap Z) = \int_Z \mu. $

By the last lemma, we have $\int_Z \mu \le (\Area Z) K^{-1+} N^{-(3 - 2 \sigma)} \le K^{+} N^{3 - \sigma} K^{-1+} N^{-(3 - 2 \sigma)} = K^{-1+} N^\sigma$.  

\end{proof}

Next we study how closely $Z$ is tangent to $v(T)$ along a reasonable segment $Seg(Q,T)$.

\begin{reasseg} \label{segtangest} $ \int_{Seg(Q,T) \cap Z} |v(T) \cdot N| \le K^{-1+}$.
\end{reasseg}

\begin{proof} For a tube $T \in \frak T$, Lemma \ref{cylinderestimate} says that

$$\int_{Z \cap T^+} | v(T) \cdot N | \le C \Deg P \lesssim K^+ N^{1 - \sigma}.$$

We consider the $\ge N$ segments $Seg(Q,T)$ where $Q$ intersects $T$.  No point lies in more than $K^{-1+} N^\sigma$ of these segments.  Therefore, for every $T \in \frak T$,

$$ \Avg_{Q \textrm{ meets } T} \int_{Z \cap Seg(Q,T)} |v(T) \cdot N| \le K^{-1+} N^{\sigma - 1} \int_{Z \cap T} |v(T) \cdot N| \le K^{-1+}. $$

So for a fraction $(1 - K^-)$ of the cubes $Q \in X$ that meet $T$, the desired estimate holds.
\end{proof}

Next we show that a reasonable tube segment contains many reasonable cubes.

\begin{reasseg} \label{manyreascubes} $Seg(Q,T)$ contains $\ge K^{-1-} N^{\sigma}$ reasonable cubes $Q'$ on each side of $Q$.
\end{reasseg}

\begin{proof}  Fix $T \in \frak T$.  Orient the tube $T$ so that one direction is `left' and the other direction is `right'.  Let $L(Q)$ be the portion of $Seg(Q,T)$ to the left of $Q$, and let $R(Q)$ be the portion of $Seg(Q,T)$ to the right of $Q$.  

Consider the set $X_{bad, left}(T)$ consisting of reasonable cubes $Q$ in $X$ so that $Q$ intersects $T$ and $L(Q)$ contains $\le K^{-1-c} N^{\sigma}$ reasonable cubes $Q'$ for a constant $c> 0$ that we'll choose below.  Consider the segments $L(Q)$ with $Q \in X_{bad, left}$.  By a Vitali-covering type argument, we can find a disjoint subset of these segments whose union contains at least a third as many reasonable cubes $Q'$ as the union of all these segments.  The segments have length $K^{-1} N^\sigma$, and they all lie in a ball of radius $K^+ N$, so the number of segments is at most $K^{1+} N^{1 - \sigma}$.  Each of these bad segments contains $\le K^{-1-c} N^{\sigma}$ reasonable cubes $Q$.  Therefore, the total number of $Q'$ lying in any bad segment $L(Q)$ is $\le K^{-c+} N$.  In particular, the number of $Q \in X_{bad, left}(T) \le K^{-c+} N$.  

Similarly, consider the set $X_{bad, right}(T)$ consisting of reasonable cubes $Q \in X$ so that $Q$ intersects $T$ and $R(Q)$ contains $\le K^{-1-c} N^{\sigma}$ reasonable cubes $Q'$.  By the same argument, $|X_{bad, right}(T)| \le K^{-c+} N$.

We let $X_{bad}(T)$ be the union of $X_{bad, left}(T)$ and $X_{bad, right}(T)$.  For each $T$, $|X_{bad}(T)| \le K^{-c+} N$.  Now we choose $c$ so that for each $T \in \frak T$, $|X_{bad}(T)| \le K^{-} N$.  

For $(1 - K^-) |\frak T|$ tubes $T \in \frak T$, a fraction $(1 - K^-)$ of the cubes $Q$ that meet $\frak T$ are reasonable.  At most $K^-$ of these cubes are in $X_{bad}$, and the remaining cubes satisify this Tube Segment Condition.

\end{proof}

We will pay particular attention to the two ends of the segment.  For each segment $Seg(Q,T)$, we choose two reasonable cubes $Q_1, Q_2$ near opposite ends of the segment.  By Reasonable Tube Segment Condition \ref{manyreascubes}, we know that $Seg(Q,T)$ contains $\ge K^{-1-} N^{\sigma}$ reasonable cubes on each side of $Q$, and so $\Dist(Q_i, Q) \ge K^{-1-} N^\sigma$.  Now by Reasonable Cube Condition \ref{plany}, we know that for each reasonable cube $Q'$, $\Avg_{T' \textrm{ intersects } (Q')^+ } \int_{(Q')^+ \cap Z} |v(T') \cdot N| \le K^{+} N^{-\sigma}$.  Now for a fraction $(1 - K^-)$ of pairs $(Q,T)$, we can choose reasonable $Q_1, Q_2$ to get the following estimate.

\begin{reasseg} \label{boundarycubes0} In each reasonable tube segment $Seg(Q,T)$, there are reasonable cubes $Q_1, Q_2$ on either side of $Q$, with $\Dist(Q_i, Q) \ge K^{-1-} N^\sigma$, obeying the integral estimate $\int_{Q_i^+ \cap Z} |v(T) \cdot N| \le K^+ N^{-\sigma}$.
\end{reasseg}

Our main goal in this Subsection is to study the geometry and regularity of $Z \cap Seg(Q,T)$.  First we study how the tangent plane of $Z$ varies, and then we study how the second fundamental form of $Z$ varies.  We are trying to prove that on each connected component of $Z \cap Seg(Q,T)$, the tangent plane and the second fundamental form of $Z$ are mostly close to constant.  

We begin with the tangent plane.
Let $w$ be a unit vector in $\RR^3$.  Recall the set $Tan(w) := \{ x \in Z | \nabla P(x) \cdot w = 0 \}$.  This is the set of points $x \in Z$ where $w \in T_x Z$.  Lemma \ref{tan(w)est} says that for each $w \in S^2$, the set $Tan(w) \subset Z$ is a curve of length $\le K^+ N^{3-2 \sigma}$.

We let $W$ denote a $K^{-1/4}$-net of points in $S^2$, with $|W| \sim K^{1/2}$.  We let $Tan(W) := \cup_{w \in W} Tan(w)$. The total length of $Tan(W)$ is still $\le K^{(1/2)+} N^{3 - 2 \sigma}$.  Next we consider the length of the intersection of this set with an average tube segment $Seg(Q,T)$.

We prove a general lemma about the average length of the intersection of a tube segment and a curve of length $L$.

\begin{lemma} \label{avglength} Let $\gamma \subset \RR^3$ be a curve of length $L = N^{3-2 \sigma} L'$.  Then

$$ \Avg_{Q \in X , T \in \frak T, Q \textrm{ meets } T} \Length Seg(Q,T) \cap \gamma \le K^{-1+} L'. $$

\end{lemma}

\begin{proof} The left-hand side is $\int_{\gamma} \mu$.  By Lemma \ref{avgbound}, this is $\le \Length(\gamma) K^{-1+} N^{-3 + 2 \sigma} = K^{-1+} L'$.  \end{proof}

Combining this lemma with our estimate that the length of $Tan(W)$ is $\le K^{(1/2)+} N^{3 - 2 \sigma}$, we get:

\begin{reasseg} The length of $Seg(Q,T) \cap Tan(W) \le K^{-(1/2)+}$.
\end{reasseg}

This length is much smaller than 1.  For comparison, $Z \cap Q$ has area $\ge 1$ for each reasonable cube $Q$.  It's a white lie to imagine that $Seg(Q,T) \cap Tan(W)$ is empty.  This stronger assumption would constrain how the tangent plane varies along a connected component of $Z \cap Seg(Q,T)$.  It would imply that the normal vector $N(x)$ is never perpendicular to any $w \in W$.  The set of points perpendicular to a fixed $w$ is a great circle, and the union over all $w \in W$ cuts the sphere $S^2$ into cells of diameter $\le K^{-(1/4)+}$.  Therefore, if $Seg(Q,T) \cap Tan (W)$ were empty, then on each connected component of $Z \cap Seg(Q,T)$, the tangent plane $T_x Z$ could vary by an angle at most $K^{-(1/4)+}$. 

Our next estimates have to do with the second fundamental form of $Z$.  Recall that for a smooth surface $Z \subset \RR^3$ with unit normal vector $N$, if $v,w \in T_x Z$, then the second fundamental form $A(v,w)$ is defined as

$$ A(v,w) := \nabla_v N(x) \cdot w. $$

\noindent In our case, we can take $N = |\nabla P|^{-1} \nabla P$.  If we want to highlight the point $x \in Z$, we refer to the second fundamental form at $x$ as $A_x$.  The second fundamental form encodes $\nabla N$, which tells us how $N(x)$ changes as $x$ moves along $Z$.

We will study several features of the second fundamental form: the Gauss curvature, the directions where the second fundamental form vanishes, the norm of the second fundamental form, etc.

The determinant of $A_x$ is the Gauss curvature of $Z$.  In other words, if $v_1, v_2$ is an orthonormal basis of $T_x Z$, then the Gauss curvature is the determinant of the matrix $A_x (v_i, v_j)$, $i,j = 1,2$.  A point is called Gauss flat if its Gauss curvature is zero.  We let $GFl \subset Z$ be the set of Gauss flat points.  

\begin{lemma} For a generic $P$, the set of Gauss flat points of $Z$ is contained in a curve of length $\le K^+ N^{3 - 2 \sigma}$.

\end{lemma}

\begin{proof} We have to check that the set of Gauss flat points is described by some polynomials vanishing.  We notice that $\nabla P \times e_i$, $i=1,2,3$ spans $TZ$ at each point of $Z$.  

Next, we notice that for $v, w \in T_x Z$, $\nabla_v (\nabla P) \cdot w = \nabla_v (|\nabla P| N) \cdot w$.  Because
$w \cdot N = 0$, this is $|\nabla P| \nabla_v N \cdot w = |\nabla P| A(v,w)$.  We record this as an equation:

$$ \nabla_v (\nabla P) \cdot w = |\nabla P| A(v,w). \eqno{(1)}$$

The Gauss curvature vanishes if and only if every $2 \times 2$ minor determinant of the matrix $A( \nabla P \times e_i, \nabla P \times e_j)$ vanishes, if and only if every $2 \times 2$ minor determinant of the following matrix vanishes: 

$$(\nabla_{\nabla P \times e_i} \nabla P) \cdot (\nabla P \times e_j).$$ 

These minor determinants are a finite list of polynomials of degree $\le 6 \Deg P$.

Since $P$ is generic, it is not Gaussian flat everywhere, and so the Gaussian flat points are contained in an algebraic curve of degree $\le 10 (\Deg P)^2 \le K^+ N^{2 - 2 \sigma}$.  Therefore the length of the Gaussian flat points is bounded by $K^+ N^{3 - 2 \sigma}$ as desired.
\end{proof}

Combining this length bound with Lemma \ref{avglength}, we get:

\begin{reasseg} The length of  $Gfl \cap Seg(Q,T)$ is $\le K^{-1+}$.
\end{reasseg}

Since $K^{-1+}$ is very small, this almost shows that the sign of the Gauss curvature is constant on connected components of $Z \cap Seg(Q,T)$.

A unit vector $v \in T_x Z$ is called straight if $A_x(v,v) = 0$.  The straight directions play an important role in the incidence geometry of lines and also in our story.  If $x$ has positive Gauss curvature, there are no straight directions.  If $x$ has negative Gauss curvature, there are exactly two straight directions.  If $x$ has zero Gauss curvature, there can be either one straight direction or else all directions may be straight if $A_x = 0$.  
We next want to control how the straight directions spin around as we vary $x$.

For a unit vector $w$, let $Str(w)$ be the set of $x \in Z$ so that there is a straight unit vector $v \in T_x Z$ with $v \cdot w = 0$. 

\begin{lemma} For generic $w$, $Str(w)$ is contained in a curve of length $\le K^+ N^{3 - 2 \sigma}$. 
\end{lemma}

\begin{proof} Suppose $x \in Str(w)$.  We know there is a straight unit vector $v \in T_x Z$ with $v \cdot w = 0$.  Since $v \in T_x Z$, $v \cdot \nabla P(x) = 0$.  Therefore, $v$ is proportional to $\nabla P \times w$.  
Hence a point $x \in Z$ lies in $Str(w)$ if and only if $A(\nabla P \times w, \nabla P \times w) = 0$.  Using equation $(1)$ above, this is equivalent to

$$(\nabla_{\nabla P \times w} \nabla P) \cdot (\nabla P \times w) = 0. $$

This is a polynomial of degree $\le 3 \Deg P \lesssim K^+ N^{1 - \sigma}$.

For generic $w$, not every point lies in $Str(w)$.  This follows because $Z(P)$ is not a plane, and so we can find a point $x$ with only finitely many straight directions, and a generic $w$ is not perpendicular to any of them.  Therefore, $Str(w)$ is an algebraic curve of degree $\le K^+ N^{2 - 2 \sigma}$ and length $\le K^+ N^{3 - 2 \sigma}$.
\end{proof}

Recall that $W$ is a $K^{-1/4}$-net of points in $S^2$ consisting of $K^{1/2}$ points.  We can choose $W$ generically so that the last lemma applies for each $w \in W$.  We let $Str(W) := \cup_{w \in W} Str(w)$.  The length of $Str(W)$ is still $\le K^{(1/2)+} N^{3 - 2 \sigma}$.

\begin{reasseg} The length of $Seg(Q,T) \cap Str(W) \le K^{-(1/2)+}$.
\end{reasseg}

As a white lie, suppose that $Seg(Q,T) \cap Str(W)$ and $Seg(Q,T) \cap GFl$ were both empty.  If $Seg(Q,T) \cap GFl$ is empty, then the sign of the Gauss curvature is constant on each component of $Z \cap Seg(Q,T)$.  Consider a component of $Z \cap Seg(Q,T)$ where the Gauss curvature is negative.  At each point there are two straight directions.  None of the straight directions is ever perpendicular to a point $w \in W$, and so the straight directions can only move by $\le K^{-1/4}$.

If the Gauss curvature of $A_x$ is positive, then there are no straight directions.  In this case, it's helpful to consider the eigenvectors of $A_x$.  For a non-zero vector $w$, let $Eig(w)$ be the set of $x \in Z$ so that there is a unit vector $v \in T_x Z$, with $v$ an eigenvector of $A_x$ and $v \cdot w = 0$. 

\begin{lemma} For a generic $w \in S^2$, $Eig(w)$ has length $\le K^+ N^{3 - 2 \sigma}$.  
\end{lemma}

\begin{proof} We begin with an algebraic description of when a non-zero vector is an eigenvector for $A_x$.

\begin{lemma} A non-zero vector $v \in T_x Z$ is an eigenvector for $A_x$ if and only if

$$ (\nabla_v \nabla P) \cdot (\nabla P \times v) = 0. $$

\end{lemma}

\begin{proof} Recall that $A_x$ is symmetric: $A_x(v,w) = A_x (w,v)$.  Therefore, a non-zero vector $v \in T_x Z$ is an eigenvector of $A_x$ if and only if $ A_x (v, u) = 0$ for all $u \in T_x Z$ with 
$u \cdot v = 0. $  The possible $u$ are all multiples of $\nabla P \times v$.  Therefore, $v$ is an eigenvector if and only if $A_x (v, \nabla P \times v) = 0$.  Recalling equation (1) above, this is equivalent to $ (\nabla_v \nabla P) \cdot (\nabla P \times v) = 0. $ \end{proof}

A point $x \in Z$ lies in $Eig(w)$ if and only if $\nabla P \times w$ is an eigenvector of $A_x$ if and only if

$$ (\nabla_{\nabla P \times w} \nabla P) \cdot (\nabla P \times (\nabla P \times w)) = 0. $$

This is a polynomial of degree $\le 4 \Deg P \le K^+ N^{1 - \sigma}$.  So $Eig(w)$ lies in an algebraic curve of degree $\le K^+ N^{2 - 2\sigma}$ and has length $\le K^+ N^{3 - 2 \sigma}$. 
\end{proof}

Recall that $W$ is a $K^{-1/4}$-net of points in $S^2$ consisting of $K^{1/2}$ points.  We can choose $W$ generically so that the last lemma applies for each $w \in W$.  We let $Eig(W) := \cup_{w \in W} Eig(w)$.  The length of $Eig(W)$ is still $\le K^{(1/2)+} N^{3 - 2 \sigma}$.

\begin{reasseg} The length of $Seg(Q,T) \cap Eig(W) \le K^{-(1/2)+}$.
\end{reasseg}

Finally, we prove similar results for the norm of the second fundamental form.
Recall that the norm of the second fundamental form $A$ is defined as follows.  Let $v_1, v_2$ be an orthonormal basis of $T_x Z$.  Then

$$ | A_x |^2 := \sum_{i,j=1}^2 | A_x (v_i, v_j) |^2. $$

\begin{lemma} \label{normformula} $ |A_x|^2 =  \sum_{i,j = 1}^3 |A_x (N \times e_i, N \times e_j)|^2$.  
\end{lemma}

\begin{proof} We begin by recalling some basic facts about the norm of a bilinear form.  If $B$ is a symmetric bilinear form on a finite-dimensional vector space $V$ with a Euclidean norm, then we define $|B|^2 := \sum_{i,j} |B(v_i, v_j)|^2$, where $v_i$ is an orthonormal basis of $V$.  It's a standard fact that this sum is independent of the choice of orthonormal basis.  To see this, suppose that $J: V \rightarrow V$ is an orthogonal transformation.  The matrix $B(J v_i, J v_j)$ is given by conjugating the matrix $B(v_i, v_j)$ by an orthogonal transformation, and this preserves the sum of the squares of the entries.

Now define a symmetric bilinear form $B$ on $\RR^3$ by

$$ B(v,w) := A_x (N \times v, N \times w). $$

On the one hand, $|B|^2 = \sum_{i,j =1}^3 |B(e_i, e_j)|^2 = \sum_{i,j = 1}^3 |A_x (N \times e_i, N \times e_j)|^2$.

On the other hand, we claim that $|B|^2 = |A_x|^2$.  To see this, choose an orthonormal basis $v_1, v_2, v_3$ for $\RR^3$ where $v_3 = N$, and $v_1, v_2 \in T_x Z$.  In this case, $N \times v_3$ vanishes, so $B(v_i, v_j) = 0$ if $i$ or $j$ is 3.  Hence

$$ |B|^2 = \sum_{i,j=1}^2 |B(v_i, v_j)|^2 = \sum_{i,j = 1}^2 |A_x (N \times v_i, N \times v_j)|^2. $$

But $N \times v_1, N \times v_2$ are an orthonormal basis of $T_x Z$, so this last expression is $|A_x|^2$. 

\end{proof}

\begin{lemma} For any generic number $H > 0$, the set $A(H) := \{ x \in Z \textrm{ such that } |A_x| = H \}$ lies in an algebraic curve of degree $\le 6 (\Deg P)^2 \le K^+ N^{2 - 2 \sigma}$, and so it has length $\le K^+ N^{3 - 2 \sigma}$.
\end{lemma}

\begin{proof} We expand $|A_x|^2$ in terms of $P$ and its derivatives.  For $v,w \in T_x Z$, 

$$ A(v,w) = \nabla_v N \cdot w = \nabla_v ( | \nabla P|^{-1} \nabla P) \cdot w. $$

Since $w \in T_x Z$, $\nabla P \cdot w = 0$, so  

$$ A(v,w) = |\nabla P|^{-1} \nabla_v \nabla P \cdot w. $$

Also, $\nabla_{N \times e_i} = \nabla_{|\nabla P|^{-1} \nabla P \times e_i} = |\nabla P|^{-1} \nabla_{\nabla P \times e_i}$.  

We plug these formulas into Lemma \ref{normformula}:

$$|A_x|^2 = \sum_{i,j=1}^3 |A_x(N \times e_i, N \times e_j)|^2 = \sum_{i,j = 1}^3 |\nabla P|^{-6} \left[(\nabla_{\nabla P \times e_i} \nabla P) \cdot (\nabla P \times e_j) \right]^2. $$

So $|A_x|^2 = H^2$ if and only if

$$ H^2 (\nabla P \cdot \nabla P)^3 - \sum_{i,j=1}^3 \left[(\nabla_{\nabla P \times e_i} \nabla P) \cdot (\nabla P \times e_j) \right]^2 = 0. $$

This equation is a polynomial equation of degree $\le 6 \Deg P$.  For generic $H$ this polynomial does not have $P$ as a factor, so the set $\{ x \in Z(P) \textrm{ such that } |A_x| = H \}$ is an algebraic curve of degree $\le 6 (\Deg P)^2$.  \end{proof}

We let $H > 0$ be a number that we will choose later.   We can add the following reasonable segment condition:

\begin{reasseg} \label{segA(H)bound} The length of $Seg(Q,T) \cap A(H)$ is $\le K^{-1+}$.
\end{reasseg}

We will choose a particular $H$ below, with $H \sim K^{1+} N^{-2 \sigma}$.  We will only need one $H$, but if we wanted to, we could choose $K^{1/2}$ different values $H_j$ and a reasonable condition would be that the length of $Seg(Q,T) \cap A(H_j)$ is $\le K^{-(1/2)+}$ for each of the values.

Suppose for a moment that $Seg(Q,T)$ intersected with $GFl, Str(W), Eig(W), Tan(W),$ and $A(H)$ (or $A(H_j)$) were all empty.  Then on each component of $Seg(Q,T)$, the second fundamental form of $Z$ would be highly constrained.  A technical issue is that these sets are not empty.  They are just small.  We get around this issue in the next subsection by intersecting $Seg(Q,T)$ with a plane.

We say that $Seg(Q,T)$ is a reasonable tube segment if it obeys Reasonable Tube Segment Conditions 1 - \ref{segA(H)bound}.

\subsection{Slices of reasonable tube segments}

Fix a reasonable tube segment $Seg(Q,T)$.  We will intersect the tube segment $Seg(Q,T)$ with a plane $\pi$ parallel to $v(T)$.  The intersection $Seg (Q,T) \cap \pi$ is a rectangle, and the intersection $Z \cap Seg(Q,T) \cap \pi$ is a curve $\Gamma$ in this rectangle.  This intersection reduces the dimension of our situation by one, making the geometry simpler.  Moreover, for a reasonable choice of $\pi$, $\Gamma$ will have no intersection with $Tan(W)$, $GFl$, $Str(W)$, $Eig(W)$, or $A(H)$.  After restricting to $\Gamma$, all the white lies above are true.  

We choose coordinates so that $T$ is given by the equation $x_1^2 + x_2^2 \le 1$.  By Reasonable Tube Condition \ref{tubeplaneangle}, we know that $\Angle(T_Q Z, v(T)) \le K^+ N^{-\sigma}$.  We choose the coordinates so that the $(x_1, x_3)$ plane is $K^+ N^{-\sigma}$ close to $T_Q Z$.  
 
We let $\pi(a,b)$ be the plane $x_1 + a x_2 = b$.  We choose $a$ uniformly at random in $(-1/10, 1/10)$ and we choose $b$ uniformly at random in $(-400, 400)$.  Because of the way we set up the coordinates, $\Angle (\pi(a,b), T_Q Z) \ge 1/10$ for all $(a,b)$. We state this as a lemma.

\begin{lemma} \label{anglepiabtqz} $\Angle( \pi(a,b), T_Q Z) \ge 1/10$.  
\end{lemma}

We let $\Gamma(a,b) = \pi(a,b) \cap Seg(Q,T) \cap Z$.  We say that a condition on $\Gamma(a,b)$ is reasonable if it holds with probability $\ge (1 - K^-)$.

For almost every $(a,b)$, $\pi(a,b) \cap Z(P)$ is an algebraic curve.

If $\gamma \subset Seg(Q,T)$ is a curve of length $L$, then the average over $(a,b)$ of the cardinality of $\pi(a,b) \cap \gamma$ is $\lesssim L$.  Therefore, with
probability $1 - K^-$, the intersections $\pi(a,b) \cap Tan(W), \pi(a,b) \cap GFl$, $\pi(a,b) \cap Eig(W)$, $\pi(a,b) \cap Str(W)$, and $\pi(a,b) \cap A(H)$ are all empty.  

\begin{reasslice} \label{badsetsempty} $\Gamma(a,b)$ does not intersect $Tan(W)$, $GFl$, $Str(W)$,  $Eig(W)$, or $A(H)$. 
\end{reasslice}

This condition has nice implications.  For a reasonable slice, as $x$ varies along a connected component of $\Gamma$, $T_x Z$ is constant up to angle $\le K^{-(1/4)+}$, and the sign of the Gauss curvature of $Z$ is constant.  If the Gauss curvature is negative, there are two straight directions at each point, and they vary continuously.   Since $\Gamma(a,b) \cap Str(W)$ is empty, the straight directions of $A_x$ are constant up to angle $K^{-(1/4)+}$ along each connected component.  If the Gauss curvature is positive, there are no straight directions.  There are always at least two eigenvector directions.  If $A_x$ is a multiple of the identity, then every direction is an eigenvector direction.  Such points lie in $Eig(w)$ for every $w$, and so there are no such points on $\Gamma$.   So at each point $x \in \Gamma$ with positive Gauss curvature, there are two distinct eigenvectors of $A_x$.  On each connected component of $\Gamma$, these eigenvector directions change by an angle $\le K^{-(1/4)+}$.  

In the last subsection, we proved an integral estimate for $|N \cdot v(T)|$ over $Z \cap Seg(Q,T)$.  Using Lemma \ref{intgeomavgest}, any integral estimate over $Z \cap Seg(Q,T)$ gives us a similar estimate over $\Gamma(a,b)$ for reasonable slices.  In particular, we get the following.

\begin{reasslice} \label{slicetangest} $\int_{\Gamma} |N \cdot v(T)| \le K^{-1+}$.
\end{reasslice}

\begin{proof} By Reasonable Tube Segment Condition \ref{segtangest}, $\int_{Z \cap Seg(Q,T)} |N \cdot v(T)| \le K^{-1+}$.  We apply Lemma \ref{intgeomavgest} to compute 

$$ \Avg_{a,b} \int_{\Gamma(a,b)} |N \cdot v(T)| \sim \int_{Z \cap Seg(Q,T)} |N \cdot v(T)| \le K^{-1+}. $$

So with probability $(1 - K^-)$, we have the desired estimate.

\end{proof}

Now we consider the geometry of the curve $\Gamma \subset \pi(a,b)$.  We let $N_\Gamma$ be the unit normal vector to $\Gamma$ inside $\pi(a,b)$.  We define the second fundamental form $A_\Gamma$.  (If $v,w \in T_x\Gamma$, then $A_\Gamma (v,w) = \nabla_v N_\Gamma \cdot w$.)  We continue to write $N$ for the normal vector to $Z$ and $A$ for the second fundamental form of $Z$.  We would like to use our information about $N$ and $A$ to study $N_\Gamma$ and $A_\Gamma$.  We begin by proving a standard differential geometry lemma about how $N$, $N_\Gamma$, $A$, $A_\Gamma$ are related.  Fix a point $x \in \Gamma(a,b) \in Z$.

\begin{lemma} \label{curvsubman} Suppose that $x \in \Gamma(a,b)$ and that $\Angle (\pi(a,b), T_x Z) = \alpha(x) > 0$.  Let $\phi$ be the orthogonal projection from $\RR^3$ to $\pi(a,b)$.

\begin{enumerate}

\item $N_\Gamma(x) = (\sin \alpha)^{-1} \phi(N(x))$.

\item If $v \in \pi(a,b)$, $v \cdot N_\Gamma(x) = (\sin \alpha)^{-1} v \cdot N(x)$.

\item If $v, w \in T_x \Gamma \subset T_x Z$, $A_{\Gamma, x}(v,w) = (\sin \alpha)^{-1} A_x (v,w)$.

\end{enumerate}

\end{lemma}

\begin{proof} We know that $N(x)$ is perpendicular to any $v \in T_x \Gamma \subset T_x Z$.  On the other hand,
$N(x) - \phi (N(x))$ is perpendicular to $\pi(a,b)$, and hence to any $v \in T_x \Gamma \subset \pi(a,b)$.  Therefore, $\phi (N(x))$ is perpendicular to $T_x \Gamma$.  The vectors $N(x)$ and $\phi(N(x))$ both point in the direction where $P$ is increasing.  Therefore, $N_\Gamma = \phi(N(x)) / | \phi(N(x))|$.  By trigonometry, $|\phi(N(x))| = \sin \alpha$.

Suppose $v \in \pi(a,b)$.  Plugging in (1), $v \cdot N_\Gamma(x) =  (\sin \alpha)^{-1} v \cdot \phi(N(x))$.  The difference $N(x) - \phi (N(x))$ is perpendicular to $\pi(a,b)$, so $v \cdot \phi(N(x)) = v \cdot N(x)$.  This gives $(2)$.

Suppose $v,w \in T_x \Gamma$.  

$$ A_\Gamma (v,w) = \nabla_v N_\Gamma \cdot w = \nabla_v ( (\sin \alpha)^{-1} \phi(N) ) \cdot w = $$

$$ = (\sin \alpha)^{-1} \nabla_v \phi(N) \cdot w + \nabla_v ( (\sin \alpha)^{-1} )  \phi(N) \cdot w. $$

Now $\phi(N)$ is normal to $T_x \Gamma$, so $\phi(N) \cdot w = 0$ and the second term vanishes.  For the first term, we note that $\phi$ and $w$ don't depend on $x$, and so $\nabla_v (\phi(N) ) \cdot w = \nabla_v ( \phi(N) \cdot w)$.  Now note that $\phi(N(y)) - N(y)$ is perpendicular to $w \in \pi(a,b)$ for every $y \in \Gamma$, and so $\nabla_v (\phi(N) \cdot w) = \nabla_v (N \cdot w) = (\nabla_v N) \cdot w = A(v,w)$.  So the first term simplifies to $(\sin \alpha)^{-1} A_x (v,w)$.  \end{proof}

This lemma shows that $N_\Gamma$ and $A_\Gamma$ are well behaved at $x$ as long as $\Angle (\pi(a,b), T_x Z)$ is not too small.  We next note that this angle is always fairly large for $x \in Z_{Q,nice}$.

\begin{lemma} \label{anglepiabnice}  If $Q$ is a reasonable cube and $x \in Z_{Q,nice}$, then $\Angle( \pi(a,b), T_x Z) \ge 1/12$.
\end{lemma}

\begin{proof} By Lemma \ref{anglepiabtqz}, $\Angle( \pi(a,b), T_Q Z) \ge 1/10$.  On the other hand, by Reasonable Cube Condition \ref{paramdisk}, $\Angle(T_x Z, T_Q Z) \le K^{-10+}$. 
\end{proof}

Using this bound, we can now start to control the geometry of a reasonable slice $\Gamma(a,b)$ through a point $x \in Z_{Q,nice}$.

\begin{lemma} \label{totalvarbound} Suppose that $Seg(Q,T)$ is a reasonable tube segment and $\Gamma$ is a reasonable slice of $Z \cap Seg(Q,T)$, and that there is a point $x \in \Gamma \cap Z_{Q,nice}$.  Let $\Gamma_1 \subset \Gamma$ be the component of of $\Gamma$ containing $x$.  Then $\Gamma_1$ obeys the following estimates:

\begin{enumerate}

\item $\Angle(\pi(a,b), T_{x'} Z) \ge (1/20)$ for all $x' \in \Gamma_1$. 

\item $\Gamma_1$ runs the whole length of $Seg(Q,T)$ .

\item $ \int_{\Gamma_1} | N_\Gamma \cdot v(T) | dx \le K^{-1+}. $

\end{enumerate}

\end{lemma}

\begin{proof} By Lemma \ref{anglepiabnice}, we know that $\Angle(\pi(a,b), T_x Z) \ge 1/12$.  

We know that $\Gamma \cap Tan(W)$ is empty.  Therefore, for all $x' \in \Gamma_1$, the tangent plane $T_{x'} Z$ is within an angle $K^{-(1/4)+}$ of $T_x Z$.  Therefore, $\Angle( \pi(a,b), T_{x'} Z) \ge (1/20)$ for all $x'$ in $\Gamma_1$.

Now that we have transversality, we can bound

$$ \int_{\Gamma_1} |N_\Gamma \cdot v(T)| \le K^+ \int_{\Gamma} |N \cdot v(T)| \le K^{-1+}. $$

Because of this integral estimate, the total variation of $\Gamma_1$ perpendicular to $v(T)$ is $\le K^{-1+}$.  Since the point $x$ lies in $Z_{Q,nice}$, which is well within the boundary of $Seg(Q,T)$, the curve $\Gamma_1$ must run the whole length of $Seg(Q,T)$.  

\end{proof}

We let $Q_1, Q_2$ be the two reasonable cubes at opposite ends of $Seg^+(Q,T)$ described in Reasonable Tube Segment Condition \ref{boundarycubes0}.

\begin{reasslice} \label{boundarycubes}

$$ \int_{\Gamma(a,b) \cap (Q_i)^+} |N \cdot v(T)| \le K^+ N^{-\sigma}. $$

\end{reasslice}

\begin{proof} By Reasonable Tube Segment Condition \ref{boundarycubes0}, we have $\int_{Z \cap Q_i^+} |N \cdot v(T) | \le K^+ N^{-\sigma}$.  Then we average using Lemma \ref{intgeomavgest}.  We note that $Z \cap Seg(Q,T) \cap \pi(a,b) = \Gamma(a,b)$, and we get 

$$ \Avg_{(a,b)} \int_{\Gamma(a,b) \cap Q_i^+} |N \cdot v(T)| \sim
\int_{Seg(Q,T) \cap Z}  |N \cdot v(T)| \le K^+ N^{-\sigma}. $$

So with probability $(1 - K^-)$ in $(a,b)$, the desired estimate holds.  \end{proof}

If $Seg(Q,T)$ is a reasonable tube segment, then we say that a slice $\Gamma(a,b)$ is a reasonable slice if it obeys Reasonable Slice Conditions 1 - \ref{boundarycubes}.

\subsection{Curvature estimates in non-straight directions}

For a given point $x \in Z$, a unit vector $v \in T_x Z$ is called straight if $A_x(v,v) = 0$.  If $x$ is not a flat point, then it has at most four straight unit vectors.   (The unit vectors come in pairs $\pm v$, and there are at most two such pairs.)  We will be interested in how far $v(T)$ is from being straight, at points $x \in Z$.  Roughly speaking, if a direction $v$ is ``far from straight'', then $|A_x(v,v)| \sim |A_x|$.  

For $x \in Z$, $v \in T_x Z$, $|v|=1$, define 

$$S_1(x,v):= \min_{w \in T_x Z, |w| = 1, w \textrm{ straight }} |v-w|. $$

This measures the angle from $v$ to a straight direction.  If $A_x$ has negative Gauss curvature, then we will prove below that when $S_1(x,v) \sim 1$, then $|A_x(v,v)| \sim |A_x|$.  If $A_x$ has positive Gauss curvature, then there are no straight directions, but there could still be a direction $v$ where $|A_x(v,v)|$ is much smaller than $|A_x|$.  In the positive Gauss curvature case, we measure the angle from $v$ to an eigenvector.  Recall that $A_x$ is called umbilic if it has two equal eigenvalues - in other words, if $A_x(v,v) = \lambda v \cdot v$ for some real number $\lambda$.  If $A_x$ is not umbilic, then it has two exactly two eigenvectors.  

$$S_2(x,v) := \min_{w \in T_x Z, |w|=1, w \textrm{ an eigenvector of } A_x} |v-w|. $$

Finally, for $x \in Z$, $v \in T_x Z$, $|v| = 1$, define $S(x,v)$ as follows:

\begin{itemize}

\item If the Gauss curvature of $Z$ at $x$ is negative, then $S(x,v) = S_1(x,v)$.

\item If the Gauss curvature of $Z$ at $x$ is non-negative and $A_x$ is not umbilic, then $S(x,v) = S_2(x,v)$.

\item If $A_x$ is umbilic, then $S(x,v) = 1$ for all $v$.

\end{itemize}

The point of this definition is that when $S(x,v)$ is not close to 0, $|A_x(v,v)|$ is comparable to $|A_x|$.
Informally, controlling $A_x$ in a non-straight direction controls $A_x$ in all directions.  We now state this precisely.

\begin{lemma} \label{totalvsdirv} For any $x \in Z$ and any unit vector $v \in T_x Z$, 

$$|A_x| \le 100 S(x,v)^{-2} |A_x(v,v)|. $$

\end{lemma}

\begin{proof}  We give slightly different proofs in the case of negative and non-negative Gauss curvature.

In the case of negative Gauss curvature, we can write the second fundamental form as a product of two linear functions: $A_x(v,v) = L_1(v) L_2 (v)$, where $L_1, L_2$ are linear maps from $T_x Z$ to $\RR$.  We let $|L_i|$ denote the maximum of $|L_i(v)|$ over all vectors $v \in T_x Z$ with $|v| \le 1$.  For each $i$, we have $|L_i| \le 2 S(x,v)^{-1} |L_i(v)|$.  Therefore, $|A_x| \le 5 |L_1| |L_2| \le 20 S(x,v)^{-2} |L_1(v)| |L_2(v)| = 20 S(x,v)^{-2} |A_x(v,v)|$.

Suppose that $Z$ has non-negative Gauss curvature at $x$ and that $A_x$ is not umbilic.  Then there are unit eigenvectors $v_1, v_2$ for $A_x$.  We have $A_x(v_i, v_i) = \lambda_i$, and $A_x(v_i, v_j) = 0$ for $i \not= j$.  Because the Gauss curvature is non-negative, the two $\lambda_i$ have the same sign (or else one of them vanishes).  The vector $v$ can be written as $v = a v_1 + b v_2$, where $|a|, |b| \ge (1/5) S(x,v)$.  Now $|A_x(v,v)| = |a^2
\lambda_1 + b^2 \lambda_2|$.  Because the eigenvalues have the same sign, this is $\ge \min(a^2, b^2) \max (\lambda_1, \lambda_2) \ge (1/100) S(x,v)^2 |A_x|$.  

Finally, if $A_x$ is umbilic, then $A_x(v,v) = \lambda v \cdot v$, and $|A_x| = \sqrt2 \lambda$, so we have $|A_x| \le 100 |A_x(v,v)|$ for every unit vector $v$.

\end{proof}

Recall that we defined $A(H)$ to be the set of points $x \in Z(P)$ where $|A_x| = H$.  We proved that for a reasonable slice, $\Gamma \cap A(H)$ is empty for a particular value $H$ that we would choose later.  In the proof of the next lemma, we will choose this $H < K^{1+} N^{-2 \sigma}$, and we will prove that along reasonable slices ``in non-straight directions'',  the second fundamental form is bounded by $H$.

\begin{lemma} \label{curvboundnonstraightdir} Suppose that $Seg(Q,T)$ is a reasonable tube segment and $\Gamma$ is a reasonable slice of $Z \cap Seg(Q,T)$, and that there is a point $x \in \Gamma \cap Z_{Q,nice}$ where $S(x, v(T)) \ge K^-$.  Let $\Gamma_1 \subset \Gamma$ be the component of of $\Gamma$ containing $x$.  Then at every point $x'  \in \Gamma_1$, we have 

$$|A_{x'}| < H \le K^{1+} N^{-2 \sigma}. $$

\end{lemma}

\begin{proof} By Lemma \ref{totalvarbound}, we  know that $\Gamma_1$ runs the whole length of $Seg(Q,T)$, and so $\Length( \Gamma_1) \ge K^{-1} N^\sigma$.  Lemma \ref{totalvarbound} also tells us that

$$\Angle(\pi(a,b), T_{x'}Z) \ge (1/20) \textrm{ for all } x' \in \Gamma_1. $$

Also, by Reasonable Slice Condition \ref{slicetangest}, we know that

$$ \int_{\Gamma_1} |N \cdot v(T)| \le  K^{-1+}. $$

\noindent Because of this integral estimate, we can find $x' \in \Gamma_1$ where $|N(x') \cdot v(T)| \le K^+ N^{-\sigma}$.   We know that $\Gamma_1$ is disjoint from $Tan(W)$ and so the normal vector $N(x')$ varies by at most $K^{-(1/4)+}$ along $\Gamma_1$.  Therefore, $\Angle( T_{x'} \Gamma_1, v(T)) \le K^{-(1/4)+}$ for all $x' \in \Gamma_1$.  Let $v_1(x)$ be the unit tangent vector in $T_x \Gamma_1$ that is roughly parallel to $v(T)$.  (There are two unit tangent vectors at each point, one roughly parallel to $v(T)$ and one roughly parallel to $- v(T)$.)  We know that $|v_1(x) - v(T)| \le K^{-(1/4)+}$. 

Our next goal is to prove that $S(x', v_1(x')) \ge K^-$ for all $x' \in \Gamma_1$.  We know that $\Gamma_1$ does not intersect $GFl$, so the Gauss curvature of $Z$ on $\Gamma_1$ is either everywhere positive or everywhere negative.  

In the negative case, there are two straight directions at each point of $\Gamma_1$, and they vary continuously.  
We know $S(x, v(T)) \ge K^-$, and so $v(T)$ is a distance $\ge K^-$ from any straight direction of $A_x$.  
Since $\Gamma \cap Str(W)$ is empty, the straight directions along $\Gamma_1$ only vary by $\le K^{-(1/4)+}$, and so $v(T)$ is a distance $\ge K^-$ from any straight of direction of $A_{x'}$.  Since $|v_1(x') - v(T)| \le K^{-(1/4)+}$, we conclude that $S(x',v_1(x')) \ge K^-$ for every $x' \in \Gamma_1$.  

In the positive case, we consider the eigenvectors instead of the straight directions.  An umbilic point lies in $Eig(w)$ for every unit vector $w$.  Since $\Gamma \cap Eig(W)$ is empty, $A_{x'}$ has two distinct eigenvectors at each point $x' \in \Gamma_1$.  These two eigenvectors vary continuously along $x'$.  We know $S(x, v(T)) \ge K^-$, and so $v(T)$ is a distance $\ge K^-$ from any any eigenvector of $A_x$.  Since $\Gamma \cap Eig(W)$ is empty, the eigenvectors only vary by an angle $\le K^{-(1/4)+}$, and so $v(T)$ is a distance $\ge K^-$ from any eigenvector of $A_{x'}$.  Since $|v_1(x') - v(T)| \le K^{-(1/4)+}$, we conclude that $S(x',v_1(x')) \ge K^-$ for every $x' \in \Gamma_1$.  

Lemma \ref{totalvsdirv} now gives us the following estimate for every $x' \in \Gamma_1$, 

$$ |A_{Z,x'} | \le K^+ |A_{Z,x'} (v_1, v_1) | . $$

Also, since $v_1(x')$ is never straight, we see that the sign of $A_{Z,x'} (v_1, v_1)$ is constant along $\Gamma_1$.  

By Lemma \ref{curvsubman}, the sign of $A_{\Gamma, x'} (v_1, v_1)$ is also constant along $\Gamma_1$.  Combining Lemma \ref{curvsubman} with the estimate $\Angle(\pi(a,b), T_{x'} Z) \ge (1/20)$ above, we see $|A_{Z,x'} (v_1, v_1)| \lesssim |A_{\Gamma,x'} (v_1, v_1)|$.  

For points $x_1, x_2 \in \Gamma_1$, define $\Gamma_1(x_1, x_2) \subset \Gamma_1$ as the segment of $\Gamma_1$ with endpoints $x_1, x_2$.  For any $x_1, x_2$, we now have the following integral estimate:

$$ \int_{\Gamma_1(x_1, x_2)} \left| A_{Z,x'} \right| \le K^+ \int_{\Gamma_1(x_1, x_2)} \left|A_{Z,x'} (v_1, v_1) \right| \le K^{+}   \left| \int_{\Gamma_1(x_1, x_2)} A_{\Gamma, x'} (v_1, v_1) \right|. $$

This last integral $\int_{\Gamma_1} A_{\Gamma,x}(v_1, v_1)$ measures the (angular) change in the unit normal vector $N_\Gamma$ from one end of $\Gamma_1$ to the other.  In particular, $\left| \int_{\Gamma_1(x_1, x_2)} A_{\Gamma, x'} (v_1, v_1) \right| \le (\pi/2) \left| N_\Gamma(x_1) - N_\Gamma(x_2) \right|$.  Putting it all together, we now have:

$$  \int_{\Gamma_1(x_1, x_2)} |A_{Z,x'}| \le K^+  | N_\Gamma(x_1) - N_\Gamma(x_2)|. \eqno{(*)}$$

Now we choose $x_1, x_2$ judiciously.
We let $Q_1, Q_2$ be the two reasonable cubes at opposite ends of $Seg(Q,T)$ described in Reasonable Tube Segment Condition \ref{boundarycubes0} and Reasonable Slice Condition \ref{boundarycubes}.  The Reasonable Slice Condition \ref{boundarycubes} says that
$\int_{\Gamma(a,b) \cap Q_i^+} |v(T) \cdot N| \le K^+ N^{-\sigma}$.  Since $\Gamma_1$ runs the whole length of $Seg(Q,T)$, we see that $\Gamma_1 \cap Q_i^+$ has length $\ge 1$ for each $i$.   Now, on $\Gamma_1$, we know that $\Angle (T_{x'} Z, \pi(a,b)) \ge 1/20$, and so $|v(T) \cdot N_\Gamma| \lesssim |v(T) \cdot N|$.  Therefore, we get

$$ \int_{\Gamma_1 \cap Q_i^+} \left| v(T) \cdot N_\Gamma \right| \le K^+ N^{-\sigma}. $$

Now we can choose $x_i \in \Gamma_1 \cap Q_i^+$ where $|v(T) \cdot N_\Gamma(x_i)| \le K^+ N^{-\sigma}$.  This implies that $|N_\Gamma(x_1) - N_\Gamma(x_2)| \le K^+ N^{-\sigma}$.  (Each vector $N_\Gamma(x_i)$ is almost normal to $v(T)$ and lies in $\pi(a,b)$.  These normal vectors cannot point in nearly opposite directions because the change in $N(x_i)$ along $\Gamma_1$ is $\le K^{-(1/4)+}$, and so the change in $N_\Gamma$ is $\le K^+ K^{-(1/4)+}$.)    Plugging in this estimate to the right-hand side in inequality $(*)$ we see: 

$$ \int_{\Gamma_1(x_1, x_2)} |A_{Z,x}|  \le K^+ |N_\Gamma(x_1) - N_\Gamma(x_2)| \le K^+ N^{-\sigma}. $$

Reasonable Tube Segment Condition \ref{boundarycubes0} says that the distance from $Q_i$ to $Q$ is $\ge K^{-1-} N^\sigma$.  So we see that $\Length(\Gamma_1(x_1, x_2) \ge K^{-1-} N^\sigma$.  Therefore, we can find a point $x' \in \Gamma_1(x_1, x_2)$ where $|A_{x'}| \le K^{1+} N^{-2 \sigma}$.

At this point, we choose the number $H$ in Reasonable Tube Segment Condition \ref{segA(H)bound}.  We choose $H$ so that

$$ |A_{x'}| < H < K^{1+} N^{-2 \sigma}. $$

By Reasonable Slice Condition \ref{badsetsempty}, $\Gamma_1 \cap A(H)$ is empty.   We conclude that $|A| < H  < K^{1+} N^{-2 \sigma}$ everywhere on $\Gamma_1$.  This proves the lemma.

\end{proof}

\subsection{Pointwise curvature bounds}

The tools from the last section allow us to prove strong bounds on the curvature of $Z$.  We will prove that at many places, the second fundamental form of $Z$ is bounded by  $K^{1+} N^{- 2 \sigma}$.  If this were true at every point, then it would instantly imply that if $x, x'$ are endpoints of a curve in $Z$ of length $\le N^{\sigma}$, then $\Angle(T_x Z, T_{x'} Z) \le K^+ N^{-\sigma}$.  Although the curvature bound does not hold at every point, we will prove that it holds in lots of places and this is sufficient to control the twisting of the tangent plane along most tube segments. 

We call $Q \in X$ a very reasonable cube if $Q$ is a reasonable cube and if a fraction $(1 - K^-)$ of the segments $Seg(Q,T)$ are reasonable.  The number of very reasonable cubes of $X$ is still $\ge (1- K^-) |X|$.  

\begin{prop} \label{goodcurvbound} If $Q$ is very reasonable, 

$$ \Area \{ x \in Z_{Q,nice} \textrm{ such that } |A_x| > H \} \le K^- \Area Z_{Q,nice}. $$

\end{prop}

Recall that $H \sim K^{1+} N^{-2 \sigma}$ was defined in the proof of Lemma \ref{curvboundnonstraightdir}.

\begin{proof} Let $Q$ be a very reasonable cube.  By the fourth item in Hypotheses \ref{uniform3trans}, we can choose three tubes $T_1, T_2, T_3$ meeting $Q$ with pairwise angles $\ge K^-$ and so that all segments $Seg(Q, T_i)$ are reasonable.  At each point $x \in Z_{Q, nice}$, $\max_{i=1}^3 S(x, v(T_i)) \ge K^-$.  (If $Z$ has negative Gauss curvature at $x$, then $S(x,v)$ measures the distance from $v$ to the straight directions of $A_x$.  Up to sign, there are only two straight directions.  Therefore, one of the three tubes must be at an angle $\ge K^-$ from straight.  The case of non-negative Gauss curvature is similar with the eigenvectors instead of the straight directions.)

Let $H < K^{1+} N^{- 2 \sigma}$ be the number chosen in the proof of Lemma \ref{curvboundnonstraightdir}.

$$ Bad_i := \{ x \in Z_{Q, nice} | S(x, v(T_i) ) > K^- \textrm{ and } |A_x| > H \}. $$

It now suffices to prove that $\Area Bad_i \le K^- \Area Z_{Q, nice}$ for each $i$.  We fix $i$ for the rest of the proof.  We consider slices of $Seg(Q, T_i)$.  

Lemma \ref{curvboundnonstraightdir} says that if $\Gamma(a,b)$ is a reasonable slice of $Z \cap Seg(Q,T_i)$, and $x \in Z_{Q,nice} \cap \Gamma(a,b)$ and $S(x, v(T_i)) \ge K^-$ then $|A_x| < H$.  Therefore, for a reasonable $\Gamma(a,b)$, $\Gamma(a,b) \cap Bad_i$ is empty.  Since a slice $\Gamma(a,b)$ is reasonable with probability $(1 - K^-)$, we get the following probability estimate:

$$ \Prob_{a,b} [ \pi(a,b) \cap Bad_i \textrm{ is non-empty} ] \le K^-. $$

We would like to use this probability estimate to bound the area of $Bad_i$.  To do this, we have to exploit the geometry of $Z_{Q, nice}$ described by nicely parametrized disks with
small holes.  We state the result we need as a lemma.

\begin{lemma} \label{sliceest1} Suppose that $X$ is an open subset of $Z_{Q,nice}$, then

$$ \Area X \le C \Prob_{a,b} [ \pi(a,b) \cap X \textrm{ is non-empty }] \Area Z_{Q, nice}. $$

\end{lemma}

\begin{proof} Recall that $Z_{Q, nice}$ is contained in the union of the graphs of some functions $f_j: B^2(10) \rightarrow \RR$ obeying $Lip(f_j) \le K^{-10+}$. Let $X_j$ be the part of $X$ in the graph of $f_j$.  Recall that the whole graph of $f_j$ has area $\sim 1$, and that almost all of the graph of $f_j$ lies in $Z_{Q,nice}$.  So it suffices to prove:

$$ \Area X_j \le C \Prob_{a,b} [ \pi(a,b) \cap X_j \textrm{ is non-empty}]. $$

Using Lemma \ref{intgeomavgest} from integral geometry, we get

$$ \Area X_j \le C \Avg_{(a,b)} \Length(X_j \cap \pi(a,b)).  $$

By Lemma \ref{anglepiabnice}, we know that $\Angle( \pi(a,b), T_x Z) \ge 1/12$ for all $x \in Graph(f_j)$.  This implies that $\pi(a,b) \cap Graph(f_j)$ is a single curve of length $\le C$.  Returning to the last inequality, we can now continue:

$$ \Area X_j \le C \Avg_{(a,b)} \Length(X_j \cap \pi(a,b)) \le C \Prob_{(a,b)} [\pi(a,b) \cap X_j \textrm{ is non-empty}]. $$

\end{proof}

We finish with a pedantic point.  The sets $Bad_i$ are not necessarily open because the function $S(x,v)$ is not continuous in $x$.  But $S(x,v)$ is continuous in $x$ on the complement of the set of Gauss flat points and the set of totally umbilic points.  So $Bad_i$ is contained in an open set and an algebraic curve.  The area of the open set is bounded by Lemma \ref{sliceest1}.  This finishes the proof of Proposition \ref{goodcurvbound}. \end{proof}

\subsection{The end of the proof}

Finally, we can bound the curvature and the change of the tangent plane along a reasonable slice and prove Theorem \ref{graininess}.

\begin{prop} If $Q$ is very reasonable, and $Seg(Q,T)$ is reasonable, and $Q'$ is a reasonable cube in $Seg(Q,T)$, then $\Angle(T_Q Z, T_{Q'} Z) \le K^{+} N^{-\sigma}$.
\end{prop}

\begin{proof} Consider the set $X \subset Z_{Q, nice}$ of points $x$ where 

\begin{itemize} 

\item $|A_x| < H$.

\item $Angle(T_x Z, T_Q Z) < K^+ N^{-\sigma}$.

\end{itemize}

Recall that $H \sim K^{1+} N^{-2 \sigma}$ was defined in the proof of Lemma \ref{curvboundnonstraightdir}, and it appears in the statement of Proposition \ref{goodcurvbound}.
By the curvature bounds in Proposition \ref{goodcurvbound}, and by the bounds on the tangent plane in Reasonable Cube Condition \ref{tangplane}, the area of $Z_{Q,nice} \setminus X$ is $\le K^- \Area Z_{Q, nice}$.  (The set of $x \in \Sigma$ where $|A_x| = H$ is an algebraic curve with area zero.)

Now we consider a random $(a,b)$ and look at the slice $\Gamma(a,b)$.  We claim that with probability $c > 0$, a random slice $\Gamma(a,b)$ contains a point $x \in X$.  By Lemma \ref{sliceest1}, we have 

$$ \Area X \le C \Prob_{a,b} [ \pi(a,b) \cap X \textrm{ is non-empty }] \Area Z_{Q, nice}. $$

\noindent But $\Area X \ge (1 - K^-) \Area Z_{Q,nice}$.  Therefore, $\Prob_{a,b} [ \pi(a,b) \cap X \textrm{ is non-empty }] \ge c > 0$.

We are going to prove that with positive probability, this slice has further good properties.

With probability $(1 - K^-)$, $\Gamma(a,b)$ is reasonable.  Let $\Gamma_1$ be the component of $\Gamma$ containing $x$.   Since $x \in Z_{Q,nice}$, Lemma \ref{totalvarbound} guarantees that $\Gamma_1$ runs the whole length of $Seg(Q,T)$.  Since $\Gamma(a,b)$ is reasonable, $\Gamma_1 \cap A(H)$ is empty.  Since $|A_x| < H$ it follows that $|A| < H \le K^{1+} N^{-2 \sigma}$ at every point of $ \Gamma_1$.   So for every $x' \in \Gamma_1$,

$$\Angle (T_{x'} Z, T_x Z) \le \int_{\Gamma_1} |A| \le K^{1+} N^{-2 \sigma} \Length (\Gamma_1) \le K^+ N^{-\sigma}. $$

By hypothesis, $Q'$ is a reasonable cube in $Seg(Q,T)$.  Since $\Gamma_1$ runs the whole length of $Seg(Q,T)$, we know that $\Gamma_1 \cap (Q')^+$ has length $\ge 1$.

Since $Q'$ is reasonable, Reasonable Cube Condition \ref{tangplane} says that $\int_{Z \cap (Q')^+} \Angle( T_{x'} Z, T_{Q'} Z) \le K^+ N^{-\sigma}$.  
Now by integral geomety (Lemma \ref{intgeomavgest}), we have
$ \Avg_{(a,b)} \int_{\Gamma(a,b) \cap (Q')^+} \Angle( T_{x'} Z, T_{Q'} Z) \le K^+ N^{-\sigma}$.
In particular, with probability $(1 - K^-)$, we have

$$ \int_{\Gamma_1 \cap (Q')^+}  \Angle( T_{x'} Z, T_{Q'} Z) \le K^+ N^{-\sigma}. $$

In particular, we can choose $x' \in \Gamma_1 \cap (Q')^+$ where $\Angle ( T_{x'} Z, T_{Q'} Z) \le K^+ N^{-\sigma}$.

Finally we have

$$\Angle(T_{Q'} Z, T_Q Z) \le \Angle(T_{Q'} Z, T_{x'} Z) + \Angle(T_{x'} Z, T_x Z) + \Angle(T_x Z, T_Q Z).$$

With positive probability (in the random choice of $(a,b)$), each of these three angles is bounded by $K^+ N^{-\sigma}$.  But $\Angle(T_{Q'} Z, T_Q Z)$ does not depend on $(a,b)$, so it must be bounded by $K^+ N^{-\sigma}$.  \end{proof}

This finishes the proof of Theorem \ref{graininess}.


\begin{thebibliography}{5}

\vskip.125in

\bibitem[B]{B} J. Bourgain, On the dimension of Kakeya sets and related maximal inequalities. Geom. Funct. Anal. 9 (1999), no. 2, 256-282. 

\bibitem[BCT]{BCT} J. Bennett, A. Carbery, and T. Tao, On the multilinear restriction and Kakeya conjectures. Acta Math. 196 (2006), no. 2, 261-302.

\bibitem[D]{D} Z. Dvir, On the size of Kakeya sets in finite fields. J. Amer. Math. Soc. 22 (2009), no. 4, 1093-1097.

\bibitem[EKS]{EKS}  G. Elekes, H. Kaplan, and M. Sharir, On lines, joints, and incidences in three dimensions. J. Combin. Theory Ser. A 118 (2011), no. 3, 962Ð977.

\bibitem[G]{Gu1}  L. Guth, The endpoint case of the Bennett-Carbery-Tao multilinear Kakeya conjecture. Acta Math. 205 (2010), no. 2, 263-286. 

\bibitem[GK]{GK}  L. Guth, and N. Katz, Algebraic methods in discrete analogs of the Kakeya problem. Adv. Math. 225 (2010), no. 5, 2828-2839.

\bibitem[KLT]{KLT} N. Katz, I. Laba, and T. Tao,
An improved bound on the Minkowski dimension of Besicovitch sets in $R^3$.
Ann. of Math. (2) 152 (2000), no. 2, 383-446. 

\bibitem[S]{S} L. Santal\'o, Integral geometry and geometric probability. Second edition. With a foreword by Mark Kac. Cambridge Mathematical Library. Cambridge University Press, Cambridge, 2004

\bibitem[ST]{ST} A. Stone, and J. Tukey, Generalized "sandwich'' theorems.
Duke Math. J. 9, (1942). 356-359. 


\end{thebibliography}
\end{document}